\documentclass[a4paper]{amsart}
\usepackage{amssymb} 

\usepackage[utf8]{inputenc}
\usepackage{amssymb,amsmath,amsthm}
\usepackage{geometry}
\usepackage{color}
\usepackage[shortlabels]{enumitem}

\usepackage[hidelinks]{hyperref} 
   \hypersetup{final} 

\usepackage{mathtools}
\mathtoolsset{centercolon} 

\newtheorem{lem}{Lemma}
\newtheorem{prop}[lem]{Proposition}
\newtheorem{thm}[lem]{Theorem}
\newtheorem{cor}[lem]{Corollary}

\theoremstyle{remark}
\newtheorem{rem}[lem]{Remark}

\theoremstyle{definition}

\def\Ex{{\mathbb E}}
\def\Pr{{\mathbb P}}
\def\er{{\mathbb R}}
\def\en{{\mathbb N}}
\def\zet{{\mathbb Z}}
\def\ve{\varepsilon}

\def\Log{\operatorname{Log}}

\def\supp{\operatorname{supp}}

\def\gamman{\beta}

\makeatletter
\@namedef{subjclassname@2020}{%
  \textup{2020} Mathematics Subject Classification}
\makeatother

\title[Operator $\ell_p\to\ell_q$ norms of  random matrices with iid entries]{Operator $\ell_p\to\ell_q$ norms of  random matrices with iid entries}

\author[R. Lata{\l}a]{Rafa{\l} Lata{\l}a}
\address{Institute of Mathematics, University of Warsaw, Banacha 2, 02--097 Warsaw, Poland.}
\email{rlatala@mimuw.edu.pl}

\author[M. Strzelecka]{Marta Strzelecka}
\address{Institute of Mathematics, University of Warsaw, Banacha 2, 02--097 Warsaw, Poland.}
\email{martast@mimuw.edu.pl}

\subjclass[2020]{%
Primary 60B20;  
Secondary  60E15;  
		46B09. 
}

\begin{document}

\maketitle

\begin{abstract}
We prove that for every $p,q\in[1,\infty]$ and every random matrix $X=(X_{i,j})_{i\le m, j\le n}$ 
with iid centered entries satisfying the regularity assumption 
$\|X_{i,j}\|_{2\rho} \le \alpha \|X_{i,j}\|_{\rho}$ for every $\rho \ge 1$, 
the expectation of the operator norm of $X$ from $\ell_p^n$ to $\ell_q^m$ is comparable, 
up to a constant depending only on $\alpha$, to
\[
m^{1/q}\sup_{t\in B_p^n}\Bigl\|\sum_{j=1}^nt_jX_{1,j}\Bigr\|_{ q\wedge \Log m}
+n^{1/p^*}\sup_{s\in B_{q^*}^m}\Bigl\|\sum_{i=1}^{m} s_iX_{i,1}\Bigr\|_{ p^*\wedge \Log n}.
\]
We give more explicit formulas, expressed as exact functions of $p$, $q$, $m$, and $n$, 
for the asymptotic operator norms in the case when the entries $X_{i,j}$ are: 
Gaussian, Weibullian, log-concave tailed, and log-convex tailed. 
In the range $1\le q\le 2\le p$ we provide two-sided bounds under a weaker regularity assumption 
$(\Ex X_{1,1}^4)^{1/4}\leq \alpha (\Ex X_{1,1}^2)^{1/2}$.
\end{abstract}

\tableofcontents


\section{Introduction and main results}

Let $X=(X_{i,j})_{i\le m, j\le n}$ be an  $m\times n$ random matrix with iid entries. 
Seginer proved in  \cite{Seginer} that if the entries $X_{i,j}$ are symmetric, then the expectation 
of the spectral norm of $X$ is of the same order as the expectation of the maximum Euclidean norm of 
rows and columns of $X$. 
In this article we address a natural question: do there exist similar formulas for 
operator norms of $X$ from $\ell_p^n$ to $\ell_q^m$, 
where $p,q\in[1,\infty]$?
Recall that if $A=(A_{i,j})_{i\le m, j\le n}$ is an $m\times n$ matrix, then
\[
\|A\|_{\ell_p^n\to \ell_q^m} 
= \sup_{t\in B_p^n} \|At\|_q 
= \sup_{t\in B_p^n, s\in B_{q^*}^m} s^TAt 
= \sup_{t\in B_p^n, s\in B_{q^*}^m} \sum_{i\le m, j\le n} A_{i,j}s_it_j
\]
denotes its operator norm from $\ell_p^n$ to $\ell_q^m$; by $\rho^*$ we denote the H{\"o}lder conjugate of 
$\rho\in [1,\infty]$, i.e., the unique element of $[1,\infty]$ satisfying $\frac 1\rho+\frac 1{\rho^*}=1$, 
and by $\|x\|_{ \rho} = (\sum_{i}|x_i|^{ \rho})^{1/{ \rho}}$  we denote the $\ell_\rho$-norm of a vector $x$
(a~similar notation, $\|Z\|_\rho = (\Ex|Z|^\rho)^{1/\rho}$ is used for the $L_\rho$-norm of a 
random variable $Z$). If $p=2=q$, then   $\|A\|_{\ell_p^n\to \ell_q^m} $  is a spectral norm of $A$,
so the case $p=2=q$ corresponds to the aforementioned result by Seginer.

Let us note that bounds for 
$\Ex \|X\|_{\ell_p^n\to \ell_q^m} $  yield both tail bounds for $\|X\|_{\ell_p^n\to \ell_q^m} $ 
and bounds for $(\Ex \|X\|_{\ell_p^n\to \ell_q^m}^{\rho})^{1/\rho}$ for every $\rho\ge 1$, 
provided that the entries of $X$ satisfy a mild regularity assumption; see 
\cite[Proposition~1.16]{APSS} for more details. 
Thus, estimating the expectation of the operator norm automatically gives us  more information 
about the behaviour of the operator norm.

Not much is known about the  nonasymptotic behaviour of the operator norms of iid random matrices if 
$(p,q)\neq (2,2)$; see the introduction to article \cite{LSChevet} for an overview of the  state of the art. 
In the case when $X_{i,j}=g_{i,j}$ are iid standard $\mathcal{N}(0,1)$ random variables one may use 
the classical Chevet's inequality \cite{Ch} to derive the following two-sided bounds 
(see \cite{LSChevet} for a~detailed calculation):
\begin{align}
\label{eq:formula-Gaussians}
\Ex\bigl\|(g_{i,j})_{i\leq m,j\leq n}\bigr\|_{\ell_p^n\rightarrow\ell_q^m}
&\sim
\begin{cases}
m^{1/q-1/2}n^{1/p^*}+n^{1/p^*-1/2}m^{1/q},&p^*,q\leq 2,
\\
 \sqrt{p^*\wedge \Log n}\:n^{1/p^*}m^{1/q-1/2}+m^{1/q},&q\leq 2\leq p^*,
\\
n^{1/p^*}+ \sqrt{q\wedge \Log m}\: m^{1/q}n^{1/p^*-1/2},&p^*\leq 2\leq q,
\\
\sqrt{p^*\wedge \Log n}\: n^{1/p^*}+\sqrt{q\wedge \Log m}\: m^{1/q},&2\leq q,p^*
\end{cases}
\nonumber
\\ 
&  \sim \sqrt{p^* \wedge \Log n}\: m^{(1/q-1/2)\vee 0} n^{1/p^*} +  
\sqrt{q \wedge \Log m}\: n^{(1/p^*-1/2)\vee 0} m^{1/q},
\end{align}   
where
\[
\Log n=\max\{1,\ln n\},
\] 
and for two nonnegative functions $f$ and $g$ we write $f\gtrsim g$ (or $g\lesssim f$) 
if there exists an absolute constant $C$ such that $Cf \ge g$;
the notation $f \sim g$ means that $f \gtrsim g$ and $g \gtrsim f$. 
We write $\lesssim_{ \alpha}$, $\sim_{K,\gamma}$, etc.\ if the underlying constant depends 
on the parameters given in the subscripts.
Equation \eqref{eq:formula-Gaussians} yields that for $n=m$ we have
\[
\Ex\bigl\|(g_{i,j})_{i,j=1}^n\bigr\|_{\ell_p^n\rightarrow\ell_q^n}\sim
\begin{cases}
n^{1/q+1/p^*-1/2},&p^*,q\leq 2,
\\
 \sqrt{p^*\wedge q\wedge \Log n}\: n^{1/(p^*\wedge q)},&  p^*\vee q\geq 2.
\end{cases}
\]

However, even in the case of exponential entries it was initially not clear for us what the order 
of the expected operator norm is. This question led us to deriving  in \cite{LSChevet} two-sided 
Chevet type bounds for iid exponential and, more generally, Weibull random vectors with shape 
parameter $r\in[1,2]$. 
In consequence, we obtained the desired nonasymptotic behaviour of operator norm in the Weibull case 
when $r\in[1,2]$ ($r=1$ is the exponential case). 
Note that this does not cover the case of a matrix $(\ve_{i,j})_{i,j}$ with iid Rademacher entries, which 
corresponds to the case $r=\infty$.
It is well known (by \cite{Bennett,CarlMaureyPuhl}, cf. \cite[Remark 4.2]{APSS})  that in this case 
 \begin{equation}
\label{eq:Radsimple}
\Ex\bigl\|(\ve_{i,j})_{i\leq m,j\leq n}\bigr\|_{\ell_p^n\rightarrow\ell_q^m}\sim_{p,q}
\begin{cases}
m^{1/q-1/2}n^{1/p^*}+n^{1/p^*-1/2}m^{1/q},&p^*,q\leq 2,
\\
m^{1/q-1/2}n^{1/p^*}+m^{1/q},&q\leq 2\leq p^*,
\\
n^{1/p^*}+n^{1/p^*-1/2}m^{1/q},&p^*\leq 2\leq q,
\\
 n^{1/p^*}+m^{1/q},&2\leq p^*,q.
\end{cases}
\end{equation}
Moreover, it is not hard to show that  constants in lower bounds do not depend on $p$ and $q$,
whereas \cite[Lemma~172]{N} shows that in the case of square matrices the constants in 
\eqref{eq:Radsimple} may be chosen to be independent of $p$ and $q$, i.e.,
\[
\Ex\bigl\|(\ve_{i,j})_{i,j=1}^n\bigr\|_{\ell_p^n\rightarrow\ell_q^n}
\sim
\begin{cases}
n^{1/q+1/p^*-1/2},&p^*,q\leq 2,
\\
n^{1/ (p^*\wedge q)},&  p^*\vee q \geq 2.
\end{cases}
\]
It is natural to ask if  the upper bound in \eqref{eq:Radsimple} does not depend on $p$ and $q$ 
also in the rectangular case. Surprisingly, the answer to this question  is negative --- 
in Corollary~\ref{cor:Radrect} below we provide an exact two-sided bound (different than the one 
in \eqref{eq:Radsimple}) up to a constant non-depending on $p$ and $q$.

The two-sided bounds for operator norms in all the aforementioned special cases may be expressed 
in the following common form:
\begin{align*}
\Ex\bigl\|(X_{i,j})_{i\leq m, j\leq n}\bigr\|_{\ell_p^n\rightarrow\ell_q^m}
 \sim
m^{1/q}\sup_{t\in B_p^n}\Bigl\|\sum_{j=1}^nt_jX_{1,j}\Bigr\|_{ q\wedge \Log m}
+n^{1/p^*}\sup_{s\in B_{q^*}^m}\Bigl\|\sum_{i=1}^{m} s_iX_{i,1}\Bigr\|_{ p^*\wedge \Log n}.
\end{align*}
Therefore, it is natural to ask if this formula is valid for other distributions of entries. 
We are able to prove it for the class of random variables $X_{i,j}$ satisfying
the following mild
regularity condition
\begin{equation}
\label{alphareg}
\|X_{i,j}\|_{2\rho}\leq \alpha \|X_{i,j}\|_{\rho} \quad \mbox{for all }  \rho\geq 1.
\end{equation}
This class contains, among others, Gaussian, Rademacher, log-concave, and Weibull random variables 
with any parameter $r\in(0,\infty)$.   Condition \eqref{alphareg} may be 
rephrased in  terms of tails of random variables $X_{i,j}$ (see Proposition \ref{prop:equiv-reg-N}).

 The main result of this paper is the following two-sided bound. 

\begin{thm}
\label{thm:rect}
Let $(X_{i,j})_{i,j\leq n}$ be iid centered random variables satisfying regularity condition \eqref{alphareg}
and let $p,q\in [1,\infty]$. Then
\begin{align*}
\Ex\bigl\|(X_{i,j})_{i\leq m, j\leq n}\bigr\|_{\ell_p^n\rightarrow\ell_q^m}
 \sim_{\alpha}
m^{1/q}\sup_{t\in B_p^n}\Bigl\|\sum_{j=1}^nt_jX_{1,j}\Bigr\|_{ q\wedge \Log m}
+n^{1/p^*}\sup_{s\in B_{q^*}^m}\Bigl\|\sum_{i=1}^{m} s_iX_{i,1}\Bigr\|_{ p^*\wedge \Log n}.
\end{align*}
\end{thm}

\begin{rem}
If $q\le 2\le p$, then the assertion of Theorem~\ref{thm:rect} holds under a weaker condition that 
random variables $X_{i,j}$ are  independent, centered,  have equal variances, and satisfy
$\|X_{i,j}\|_4\le \alpha \|X_{i,j}\|_2$. We prove this in Subsection~\ref{seq:case-leq2}.
\end{rem}

\begin{rem}\label{rmk:non-centered}
In the case when random variables $X_{i,j}$ are not necessarily  centered, Theorem~\ref{thm:rect} and Jensen's inequality  imply that (see Subsection~\ref{sect:noncentered} for a detailed proof)
\begin{align}	
\label{eq:noncentered-aim}
\notag
\Ex\bigl\|(X_{i,j})_{i, j}\bigr\|_{\ell_p^n\rightarrow\ell_q^m}
 & \sim_{\alpha}
m^{1/q}n^{1/p^*}|\Ex X_{1,1}|+
m^{1/q}\sup_{t\in B_p^n}\Bigl\|\sum_{j=1}^nt_j(X_{1,j}-\Ex X_{1,1})\Bigr\|_{ q\wedge \Log m}
\\ & \qquad\quad +n^{1/p^*}\sup_{s\in B_{q^*}^m}\Bigl\|\sum_{i=1}^{m} s_i(X_{i,1}-\Ex X_{1,1})\Bigr\|_{ p^*\wedge \Log n}
 \end{align}
provided that iid random variables $X_{i,j}$, $i\le m, j\le n$, satisfy
\begin{equation} 
\label{eq:noncentered-assumpt}
\|X_{i,j}-\Ex X_{i,j}\|_{2\rho}\leq \alpha \|X_{i,j}-\Ex X_{i,j}\|_{\rho} \quad \mbox{for all }  \rho\geq 1.
\end{equation}
\end{rem}

The formula in Theorem~\ref{thm:rect} looks quite simple but, because of the suprema appearing in it, 
it is not always easy to see how the right-hand side depends on $p$ and $q$.
In Section~\ref{sect:examples} we give exact formulas for quantities comparable to the one from 
Theorem~\ref{thm:rect} in the case when the entries are Weibulls (this includes exponential and 
Rademacher random variables) or, more generally, when the entries have log-concave or log-convex tails.

The next proposition reveals how the two-sided bound from Theorem~\ref{thm:rect} depends on $p$ and $q$ 
in the case when  $n=m$ and  $p^*\vee q \ge 2$.

\begin{prop}
\label{prop:squarelargep*}
Let $p,q\in [1,\infty)$ and $p^*\vee  q\geq 2$. Let $X_{i,j}$ be iid centered random variables
satisfying \eqref{alphareg}. Then
\[
n^{1/q}\sup_{t\in B_p^n}\Bigl\|\sum_{j=1}^nt_jX_{1,j}\Bigr\|_{ q\wedge \Log n}
+n^{1/p^*}\sup_{s\in B_{q^*}^n}\Bigl\|\sum_{i=1}^n s_iX_{i,1}\Bigr\|_{ p^*\wedge \Log n}
\sim_\alpha n^{ 1/(p^*\wedge q )}\|X_{1,1}\|_{ q\wedge \Log n}.
\]
\end{prop}

Moreover, if one of the parameters $p^*,q$ is not larger than $2$, then in the general rectangular case 
one of the terms from the formula in Theorem~\ref{thm:rect} can be simplified in the following way.

\begin{prop}
\label{rem:smallq}
For $\tilde{q}\in [1,2]$, $p\in[1,\infty)$ and centered iid random variables $X_i$  we have
\[
\frac{1}{2\sqrt{2}}n^{(1/p^*-1/2)_+}\|X_{1}\|_{\tilde{q}}
\leq
\sup_{t\in B_p^n}\Bigl\|\sum_{j=1}^nt_jX_{j}\Bigr\|_{\tilde{q}}
\leq 
n^{(1/p^*-1/2)_+}\|X_{1}\|_2.
\]
Similarly, for $\tilde{p}\in[1,2]$ and $q\in [1,\infty)$,
\[
\frac{1}{2\sqrt{2}}m^{(1/q-1/2)_+}\|X_{1}\|_{\tilde{p}}
\leq
\sup_{s\in B_{q^*}^{m}}\Bigl\|\sum_{i=1}^m s_iX_{i}\Bigr\|_{\tilde{p}}
\leq 
m^{(1/q-1/2)_+}\|X_{1}\|_2.
\]
In particular, if $1\le p^*,q\le 2$, and $X_{i,j}$'s are iid random variables satisfying 
$\widetilde{\alpha}^{-1}\|X_{i,j}\|_1\ge  \|X_{i,j}\|_2=1$, then 
\begin{align*}
&m^{1/q}\sup_{t\in B_p^n}\Bigl\|\sum_{j=1}^nt_jX_{1,j}\Bigr\|_{ q}
&+n^{1/p^*}\sup_{s\in B_{q^*}^m}\Bigl\|\sum_{i=1}^{m} s_iX_{i,1}\Bigr\|_{ p^*}
\sim_{\widetilde{\alpha}} m^{1/q}n^{1/p^*-1/2}+n^{1/p^*}m^{1/q-1/2}.
\end{align*}
\end{prop}

Theorem~\ref{thm:rect} and the last part of Proposition~\ref{rem:smallq} imply that under 
the regularity assumption \eqref{alphareg} the behaviour of 
$\Ex\|(X_{i,j})_{i,j=1}^n\|_{\ell_p^n\rightarrow\ell_q^n}$ in the range $1\le p^*,q\le 2$ is 
the same as in the case of iid Gaussian matrix (see \eqref{eq:formula-Gaussians}), whose entries 
have the same variance as $X_{1,1}$.

Propositions~\ref{prop:squarelargep*} and \ref{rem:smallq} yield that
in the case of square matrices the  bound from Theorem~\ref{thm:rect} may be expressed in a more explicit 
way in the whole range of $p$ and $q$:

\begin{cor}
\label{cor:square}
Let $(X_{i,j})_{i,j\leq n}$ be iid centered random variables satisfying regularity condition \eqref{alphareg} 
and let $1\leq p,q\leq \infty$.
Then
\[
\Ex\bigl\|(X_{i,j})_{i,j=1}^n\bigr\|_{\ell_p^n\rightarrow\ell_q^n}\sim_\alpha
\begin{cases}
n^{1/q+1/p^*-1/2}\|X_{1,1}\|_2,&p^*,q\leq 2,
\\
n^{1/(p^*\wedge q)}\|X_{1,1}\|_{ p^*\wedge q\wedge \Log n} ,&  p^*\vee q\geq 2.
\end{cases}
\]
\end{cor}

The rest of this article is organized as follows.  In Section \ref{sect:prelim} we review 
properties of random variables satisfying regularity condition \eqref{alphareg}.
In Section~\ref{sect:examples} we provide   explicit functions of parameters $p^*$, $q$, $n$, $m$ 
comparable to the bounds from Theorem~\ref{thm:rect} for some special classes of distributions, 
and prove Remark~\ref{rmk:non-centered}.
In Section~\ref{sect:lower-bounds} we  establish the lower bound of Theorem~\ref{thm:rect}, 
and in Section~\ref{sect:largergeq2} we give proofs of Propositions~\ref{prop:squarelargep*} 
and \ref{rem:smallq}.
 Section~\ref{sect:upper-bounds} contains the proof of the upper bound of Theorem~\ref{thm:rect}. It is divided into 
several subsections corresponding to particular ranges of $(p,q)$, since the arguments we use in the 
proof vary depending on the range we deal with. In Subsections \ref{sect:outline} and \ref{sect:tools} 
we reveal the methods and tools, respectively, used in the most challenging parts of the proof.



\section{Properties of $\alpha$-regular random variables}
\label{sect:prelim}

In this section we discuss crucial properties of random variables satisfying  $\alpha$-regularity condition~\eqref{alphareg}. We also show how to express this condition in terms of tails.

One of the important consequences of  $\alpha$-regularity condition \eqref{alphareg} is the 
comparison of weak and strong moments of  linear combinations of independent centered variables 
$X_{i,j}$, proven in \cite{LS}, stating that
for every $\rho\geq 1$ and every nonempty bounded $U\subset \er^{nm}$,
\begin{equation}\label{eq:comp-moments-regular}
\Bigl(\Ex\sup_{u\in U}\Bigl|\sum_{i,j}X_{i,j}u_{i,j}\Bigr|^\rho\Bigr)^{1/\rho}
\sim \Ex\sup_{u\in U}\Bigl|\sum_{i,j}X_{i,j}u_{i,j}\Bigr|
+\sup_{u\in U}\Bigl\|\sum_{i,j}X_{i,j}u_{i,j}\Bigr\|_{\rho}.
\end{equation}
Another  property of independent centered variables satisfying  \eqref{alphareg} is the following Khintchine--Kahane-type estimate,
derived in \cite[Lemma~4.1]{LS},
\begin{equation}
\label{eq:compmomregalpha}
\Bigl\|\sum_{i,j}u_{i,j}X_{i,j}\Bigr\|_{ \rho_1}\lesssim_\alpha 
\Bigl(\frac{\rho_1}{\rho_2}\Bigr)^\beta \Bigl\|\sum_{i,j}u_{i,j}X_{i,j}\Bigr\|_{\rho_2}\qquad \text{for every }  
 \rho_1\geq  \rho_2\geq 1,
\end{equation}
where $\beta:=\frac{1}{2}\vee \log_2\alpha$ and  $u$ is an arbitrary  $m\times n$ deterministic matrix. 

For  iid random variables $X_{i,j}$ we define  their log-tail function 
$N\colon [0,\infty)\to [0,\infty]$  via the formula
\begin{equation}
\label{eq:logtail}
\Pr(|X_{i,j}|\geq t)=e^{-N(t)}, \quad t\geq 0.
\end{equation}
Function $N$ is nondecreasing, but not necessary invertible. 
However, we may consider  its generalized inverse 
$N^{-1}\colon [0,\infty)\to [0,\infty)$ defined by
\[
N^{-1}(s)=\sup\{t \ge 0\colon N(t)\leq s\}.
\]

\begin{lem}
\label{lem:mom_invtail}
 Suppose that condition \eqref{alphareg} holds and $N$ is defined by \eqref{eq:logtail}.
Then for every $\rho\geq 1$,
\[
\|X_{i,j}\|_\rho\sim_\alpha N^{-1} \bigl(\rho\vee (2\ln(2\alpha))\bigr).
\]
\end{lem}

\begin{proof}
To simplify the notation set $\gamma:=2\ln(2\alpha)$. Note that  $\alpha\ge 1$ and 
$\gamma> 1$.

For $t<N^{-1}(\rho\vee \gamma)$ we have by Chebyshev's inequality
\[
\|X_{i,j}\|_\rho\geq \Pr(|X_{i,j}|\geq t)^{1/\rho}t\geq e^{-1\vee(\gamma/\rho)}t
\geq e^{-\gamma}t.
\]
Hence, $N^{-1}(\rho\vee \gamma)\leq 4\alpha^2\|X_{i,j}\|_\rho$.

To derive the opposite bound, observe that the Paley-Zygmund inequality and regularity assumption \eqref{alphareg} yield that for every $\rho\geq 1$,
\[
\Pr\Bigl(|X_{i,j}|\geq \frac{1}{2}\|X_{i,j}\|_\rho\Bigr)
=\Pr(|X_{i,j}|^\rho\geq 2^{-\rho}\Ex|X_{i,j}|^\rho)
\geq (1-2^{-\rho})^2\frac{(\Ex|X_{i,j}|^\rho)^2}{\Ex|X_{i,j}|^{2\rho}}
\geq \frac{1}{4}\alpha^{-2\rho}\geq e^{-\gamma\rho}.
\]
Therefore, $N^{-1}(\gamma\rho)\geq \frac{1}{2}\|X_{i,j}\|_\rho$ for every $\rho \ge 1$, so by taking $\rho = 1\vee (\rho/\gamma)$ and applying \eqref{alphareg} multiple times we get
\[
N^{-1}(\rho\vee \gamma)\geq \frac{1}{2}\|X_{i,j}\|_{1\vee(\rho/\gamma)}
\geq \frac{1}{2}\alpha^{-\lceil \log_2\gamma\rceil}\|X_{i,j}\|_\rho
\geq \frac{1}{2}(2\gamma)^{-\log_2 \alpha}\|X_{i,j}\|_\rho. \qedhere
\]
\end{proof}

\begin{rem}\label{rmk:N-inverse-moments}
The proof above shows that 
\[
\frac{1}{e}N^{-1}(\rho)\leq \|X_{i,j}\|_\rho\leq 2(4\ln(2\alpha))^{ \log_2 \alpha} N^{-1}(\rho)
\quad \mbox{ for }\rho\geq 2\ln(2\alpha).
\]
\end{rem}

The next proposition shows how to rephrase condition \eqref{alphareg} in terms of tails of $X_{i,j}$.

\begin{prop}
\label{prop:equiv-reg-N}
Let $X$ be a random variable and  $\Pr(|X|\geq t)=e^{-N(t)}$ for $N\colon [0,\infty)\to [0,\infty]$.
Then the following conditions are equivalent\\
i) there exists $\alpha_1\in [1,\infty)$ such that $\|X\|_{2\rho}\leq \alpha_1\|X\|_\rho$ for every $\rho\geq 1$;\\
ii) there exist $\alpha_2\in [1,\infty)$, $\beta_2 \in [0,\infty)$ such that 
$N^{-1}(2s)\leq \alpha_2N^{-1}(s)$ for every
$s> \beta_2$;\\
iii) there exist $\alpha_2\in [1,\infty)$, $\beta_2 \in [0,\infty)$  such that $N(\alpha_2t)\geq 2N(t)$ for every $t>0$ satisfying $N(t)> \beta_2$.
\end{prop}

\begin{proof}
i)$\Rightarrow$ ii) By Lemma \ref{lem:mom_invtail} we have for $s>2\ln(2\alpha_1)$,
\[
N^{-1}(2s)\sim_{\alpha_1} \|X\|_{2s}\leq\alpha_1\|X\|_s\sim_{\alpha_1} N^{-1}(s).
\]

Equivalence of ii) and iii) is standard.

iii)$\Rightarrow$ i) Let us fix $\rho\geq 1$. We have
$\|X\|_\rho^\rho\geq t^\rho\Pr(|X|\geq t)=t^\rho e^{-N(t)}$. Thus $N(t)>\beta_2$ for 
$t>t_0:=e^{\beta_2/\rho}\|X\|_\rho$, and so
\begin{align*}
\|X\|_{2\rho}^{2\rho}
&\leq 
\alpha_2^{2\rho}\Bigl(t_0^{2\rho}+2\rho\int_{t_0}^\infty t^{2\rho-1}e^{-N(\alpha_2t)}dt\Bigr)
\leq \alpha_2^{2\rho}\Bigl(t_0^{2\rho}
+2\rho\int_{t_0}^\infty t^{\rho}e^{-N(t)}t^{\rho -1} e^{-N(t)}dt\Bigr)
\\
&\leq \alpha_2^{2\rho}\Bigl(t_0^{2\rho}+2\|X\|_\rho^\rho \rho
\int_{t_0}^\infty t^{\rho -1}e^{-N(t)}dt\Bigr)
\leq \alpha_2^{2\rho}\|X\|_\rho^{2\rho}(e^{2\beta_2}+2)
\\
&\leq (\alpha_2(e^{\beta_2}+ \sqrt{2}))^{2\rho}\|X\|_{\rho}^{2\rho}.\qedhere
\end{align*}
\end{proof}

\begin{rem}
Remark~\ref{rmk:N-inverse-moments}  and the proof above show that i) implies ii) and iii) with  constants
$\alpha_2=~2e\alpha_1(4\ln(2\alpha_1))^{ \log_2 \alpha_1}$, $\beta_2=2\ln(2\alpha_1)$, 
and conditions ii), iii) imply i) with  constants $\alpha_1=\alpha_2(e^{\beta_2}+ \sqrt{2})$.
\end{rem}

\section{Examples}	
\label{sect:examples}

In this section we focus on  two particular classes of distributions: with log-concave and log-convex tails. 
They include Rademachers, subexponential Weibulls, and heavy-tailed Weibulls. 
Our aim is to provide an explicit function of parameters $p^*$, $q$, $n$, $m$ comparable to the bounds 
from Theorem~\ref{thm:rect}; such a function in the case of iid Gaussian matrices is given in 
\eqref{eq:formula-Gaussians}. 

 Throughout this  section,  we assume that $X_{i,j}$ are iid symmetric random variables and their log-tail
function $N\colon [0,\infty)\to [0,\infty]$ is given by \eqref{eq:logtail}.

\subsection{Variables with log-concave tails} 

In this subsection we consider variables with log-concave tails, i.e., variables
with convex log-tail function $N$. 
Since $N(0)=0$ and $N$ is convex,  for every $s>t>0$ we have
\begin{equation}
\label{eq:diff-quot-N}
\frac{N(s)}s \ge \frac{N(t)}t.
\end{equation}
In particular, Proposition~\ref{prop:equiv-reg-N} yields that  a random variable with log-concave tails satisfy \eqref{alphareg} with a universal constant $\alpha$.
Hence, in the square case  Corollary~\ref{cor:square} and Lemma~\ref{lem:mom_invtail} 
imply that
\begin{align*}
\Ex\bigl\|(X_{i,j})_{i,j=1}^n\bigr\|_{\ell_p^n\rightarrow\ell_q^n},
&\sim_r 
\begin{cases}
n^{1/q+1/p^*-1/2}N^{-1}(1) &    p^*,q\leq 2,\\
n^{1/(p^*\wedge q)}N^{-1}(p^*\wedge q\wedge \Log n)&   p^*\vee q \ge 2
\end{cases} 
\\
&  \sim  N^{-1}(p^*\wedge q\wedge \Log n)n^{1/(p^*\wedge q)}n^{ (1/(p^*\vee q)-1/2)\vee 0}.
\end{align*}

In the case of log-concave tails it is  more convenient to normalize random variables  in such a way that $N^{-1}(1)=1$ rather than 
$\|X_{i,j}\|_2=1$. Observe that Lemma~\ref{lem:mom_invtail} and \eqref{eq:diff-quot-N} yield that
$\|X_{i,j}\|_2\sim N^{-1}(1)$.

\begin{lem}
\label{lem:momlogconcave}
Let $X_1,\ldots,X_n$ be iid symmetric random variables  with log-concave tails such that  $N^{-1}(1)=1$.
Then for every $p,q\geq 1$,
\[
\sup_{t\in B_p^n}\Bigr\|\sum_{i=1}^nt_iX_i\Bigl\|_q
\sim\max_{1\leq k\leq q\wedge n}k^{1/p^*}N^{-1}(q/k) 
+ (q \wedge n)^{1/(p^*\vee 2)}n^{(1/p^*-1/2)\vee 0}.
\]
\end{lem}

\begin{proof}
The result of Gluskin and Kwapień \cite{GK} states that
\[
\Bigr\|\sum_{i=1}^nt_iX_i\Bigl\|_q\sim
\sup\Bigl\{\sum_{i\leq q\wedge n} t_i^*s_i\colon\ \sum_{i\leq q\wedge n}N(s_i)\leq q\Bigr\}
+\sqrt{q}\Bigr(\sum_{i>q}|t_i^*|^2\Bigl)^{1/2},
\]
where $t_1^*,\ldots, t_n^*$ is the nonincreasing rearrangement of $|t_1|,\ldots,|t_n|$.

Let us fix $t\in B_p^n$. Then for every  $q>n$, 
\[
\sum_{i\leq q}t_i^*+\sqrt{q}\Bigl(\sum_{k>q}(t_k^*)^2\Bigr)^{1/2}
=\sum_{i\leq n}t_i^*\leq n^{1-1/p}= n^{1/2-1/p} \sqrt{q\wedge n}
= (q\wedge n)^{1/p^*}.
\]
For $p\geq 2$ and $q<n$ we  have 
\begin{align*}
\sum_{i\leq q}t_i^*+\sqrt{q}\Bigl(\sum_{k>q}(t_k^*)^2\Bigr)^{1/2}
\leq q^{1-1/p}+q^{1/2}(n-q)^{1/2-1/p}\sim q^{1/2}n^{1/2-1/p}
=n^{1/2-1/p}  \sqrt{q\wedge n}.
\end{align*}
Finally, for $p\in[1, 2]$, $q<n$ we obtain
\begin{align*}
\sum_{i\leq q}t_i^*+\sqrt{q}\Bigl(\sum_{k>q}(t_k^*)^2\Bigr)^{1/2}
&\leq \sum_{i\leq q}t_i^*+\sqrt{q} (t_q^\ast)^{(2-p)/2} \Bigl(\sum_{k>q}(t_k^*)^p\Bigr)^{1/2}
\leq q^{1-1/p}+q^{1/2}(t_q^*)^{1-p/2}
\\
&\leq 2q^{1-1/p} =2 (q\wedge n)^{1/p^*}.	
\end{align*}
Estimates above might be reversed up to universal constants
if we take $t=\sum_{i=1}^n n^{-1/p}e_i$ for $p\geq 2$, and  $t=\sum_{i=1}^{q\wedge n}
 (q\wedge n)^{-1/p}e_i$ for $p\in [1,2]$.
Thus, in any case,
\[
\sup_{t\in B_p^n}\Bigl(\sum_{i\leq q\wedge n} t_i^*
+\sqrt{q}\Bigr(\sum_{i>q}|t_i^*|^2\Bigl)^{1/2}\Bigr)
\sim (q \wedge n)^{1/(p^*\vee 2)}n^{(1/p^*-1/2)\vee 0}.
\]
Moreover, since $ N^{-1}(1)=1$,
\begin{align*}
\sqrt{q}\Bigr(\sum_{i>q}|t_i^*|^2\Bigl)^{1/2} 
& \le \sum_{i\leq q\wedge n} t_i^*
+\sqrt{q}\Bigr(\sum_{i>q}|t_i^*|^2\Bigl)^{1/2}
\\ 
& \le \sup\Bigl\{\sum_{i\leq q\wedge n} t_i^*s_i\colon\ \sum_{i\leq q\wedge n}N(s_i)\leq q\Bigr\}
+\sqrt{q}\Bigr(\sum_{i>q}|t_i^*|^2\Bigl)^{1/2}.
\end{align*}

Hence, it remains to prove that
\begin{align*}
\sup_{t\in B_p^n}
\sup\Bigl\{\sum_{i\leq q\wedge n} t_i^*s_i\colon\ \sum_{i\leq q\wedge n}N(s_i)\leq q\Bigr\}
&=\sup\Bigl\{\Bigl(\sum_{i\leq q\wedge n} |s_i|^{p^*}\Bigr)^{1/p^*}\colon\ 
\sum_{i\leq q\wedge n}N(s_i)\leq q\Bigr\}
\\
&\sim \max_{1\leq k\leq q\wedge n}k^{1/p^*}N^{-1}(q/k).
\end{align*}

The lower bound is obvious since $N(N^{-1}(u))\le u$ for every $u\ge 0$. 
To show the upper estimate let
\[
a:=\max_{1\leq k\leq q\wedge n}k^{1/p^*}N^{-1}(q/k),
\]
where the maximum runs  through integers $k$ satisfying $1\leq k\leq q\wedge n$.
Then \eqref{eq:diff-quot-N} implies that
\[
\sup_{1\leq t\leq q\wedge n}t^{1/p^*}N^{-1}(q/t)\leq 2a,
\]
where the supremum runs through all $t\in \er$ satisfying $1\leq t\leq q\wedge n$.
Hence,
\[
N(s)\geq q\Bigl(\frac{s}{2a}\Bigr)^{p^*}\quad \mbox{ whenever }  2a\geq s\geq 2a (q\wedge n)^{-1/p^*}. 
\]
Therefore, condition $\sum_{i\leq q\wedge n}N(s_i)\leq q$ yields that  $s_i\le a$ and so
\[
\sum_{i\leq q\wedge n} s_i^{p^*}\leq 
(2a)^{p^*}\sum_{i\leq q\wedge n}\Bigl(\frac{1}{q\wedge n}+\frac{1}{q}N(s_i)\Bigr)
\leq 2(2a)^{p^*}\leq (4a)^{p*}.\qedhere
\]
\end{proof}

Theorem \ref{thm:rect} and Lemma \ref{lem:momlogconcave} yield the following corollary. 

\begin{cor}
\label{cor:rectlogconcave}
Let $(X_{i,j})_{i\leq m,j\leq n}$ be iid symmetric random variables with log-concave tails 
such that $N^{-1}(1)=1$. Then for every $p,q\ge 1$,
\begin{align}
\label{eq:weibull-subgaussian}
\notag
\Ex&\bigl\|(X_{i,j})_{i\leq m,j\leq n}\bigr\|_{\ell_p^n\rightarrow\ell_q^m}
\\ 
\notag
&\sim_\alpha
\begin{cases}
m^{1/q-1/2}n^{1/p^*}+n^{1/p^*-1/2}m^{1/q},&p^*,q\leq 2,
\\
n^{1/p^*}\Bigl(\sqrt{p^*\wedge m \wedge \Log n }\, m^{1/q-1/2} +
\sup\limits_{l\leq p^*\wedge \Log n\wedge m}
l^{1/q}N^{-1}\bigl(\frac{p^*\wedge\Log n}{l}\bigr)\Bigr)
+m^{1/q},&q\leq 2\leq p^*,
\\
  n^{1/p^*}  + m^{1/q}\Bigl(\sqrt{q\wedge n \wedge \Log m}\,n^{1/p^*-1/2}
 +\sup\limits_{k\leq q\wedge \Log m\wedge n}
k^{1/p^*}N^{-1}\bigl(\frac{q\wedge\Log m}k\bigr)\Bigr),
&p^*\leq 2\leq q,
\\
n^{1/p^*}\Bigl((p^*\wedge m\wedge \Log n)^{1/q}
+\sup\limits_{l\leq p^*\wedge \Log n\wedge m}l^{1/q}N^{-1}\bigl(\frac{p^*\wedge\Log n }{l}\bigr)\Bigr)
\\
\qquad +m^{1/q}\Bigl((q\wedge n\wedge \Log m)^{1/p^*}
+\sup\limits_{ k\leq q\wedge \Log m\wedge n}k^{1/p^*}N^{-1}\bigl(\frac{q\wedge\Log m }k\bigr)\Bigr) ,
&2\leq p^*,q
\end{cases}
\end{align}
\end{cor}

\subsubsection{Subexponential Weibull matrices}

Let $X_{i,j}$ be symmetric Weibull random variables with parameter $r$, i.e., $N(t)=t^r$. 
If $X_{i,j}$ are subexpenential, i.e. $r\geq 1$, then $N$ is convex, 
and 
$\|X_{i,j}\|_\rho = (\Gamma(1+\rho/r)^{1/\rho}) \sim \rho^{1/r}$. 
Thus, Corollary~\ref{cor:square} implies that
\begin{align*}
\Ex\bigl\|(X_{i,j})_{i,j=1}^n\bigr\|_{\ell_p^n\rightarrow\ell_q^n},
&\sim 
\begin{cases}
n^{1/q+1/p^*-1/2} &    p^*,q\leq 2,\\
(p^*\wedge q\wedge \Log n)^{1/r}  n^{1/(p^*\wedge q)},  &   p^*\vee q \leq 2
\end{cases} 
\\
&  \sim  (p^*\wedge q\wedge \Log n)^{1/r}n^{1/(p^*\wedge q)}n^{ (1/(p^*\vee q)-1/2)\vee 0}.
\end{align*}

To obtain a formula in the rectangular case we first observe  that $N^{-1}(1)=1$ and
\[
\sup_{1\leq k\leq l}k^{1/p^*}N^{-1}(q/k)
=q^{1/r}l^{(1/p^*-1/r)\vee 0}.
\]
If $r\in [1,2]$ then $1/p^*-1/r\leq 0$ for $p^*\geq 2$ and Corollary \ref{cor:rectlogconcave} allows to recover
the following bound from  \cite[Corollary 6]{LSChevet}. 
\begin{align*}
\Ex&\bigl\|(X_{i,j})_{i\leq m,j\leq n}\bigr\|_{\ell_p^n\rightarrow\ell_q^m}
\\ 
&\sim
\begin{cases}
m^{1/q-1/2}n^{1/p^*}+n^{1/p^*-1/2}m^{1/q},&p^*,q\leq 2,
\\
 (p^*\wedge \Log n)^{1/r}n^{1/p^*}m^{ (1/q-1/r)\vee 0}
+ \sqrt{p^*\wedge \Log n}\: n^{1/p^*}m^{1/q-1/2}+m^{1/q},&q\leq 2\leq p^*,
\\
 n^{1/p^*}+ (q\wedge \Log m)^{1/r}m^{1/q}n^{(1/p^*-1/r)\vee 0}
+ \sqrt{q\wedge \Log m}\: m^{1/q}n^{1/p^*-1/2},&p^*\leq 2\leq q,
\\
 (p^*\wedge \Log n)^{1/r}n^{1/p^*}+ (q\wedge\Log m)^{1/r}m^{1/q},&2\leq p^*,q
\end{cases}
\\  
&\sim
 (p^*\wedge \Log n)^{1/r} m^{(1/q-1/r)\vee 0}n^{1/p^*} +\sqrt{p^*\wedge \Log n}\: m^{(1/q-1/2)\vee 0} n^{1/p^*}
\\
& \qquad
+ (q\wedge \Log m)^{1/r} n^{(1/p^*-1/r)\vee 0}m^{1/q} +\sqrt{q\wedge \Log m}\: n^{(1/p^*-1/2)\vee 0} m^{1/q}.
\end{align*}

In the case $r>2$ Corollary \ref{cor:rectlogconcave} yields the following.

\begin{cor}
Let $(X_{i,j})_{i\leq m,j\leq n}$ be iid Weibull random variables with parameter $r\ge 2$. Then for every $p,q\ge 1$,
\begin{align*}
\notag
\Ex\bigl\|(X_{i,j})_{i\leq m,j\leq n}&\bigr\|_{\ell_p^n\rightarrow\ell_q^m}
\\ 
&\sim
\begin{cases}
m^{1/q-1/2}n^{1/p^*}+n^{1/p^*-1/2}m^{1/q},&p^*,q\leq 2,
\\
 m^{1/q-1/2}(p^*\wedge \Log n)^{1/r}(p^*\wedge\Log n\wedge m)^{1/2-1/r}n^{1/p^*}+m^{1/q},&q\leq 2\leq p^*,
\\
 n^{1/p^*}+n^{1/p^*-1/2}(q\wedge \Log m)^{1/r}(q\wedge\Log m\wedge n)^{1/2-1/r}m^{1/q},&p^*\leq 2\leq q,
\\
(p^*\wedge \Log n)^{1/r}(p^*\wedge\Log n\wedge m)^{(1/q-1/r)\vee 0}n^{1/p^*}
\\
\qquad+(q\wedge \Log m)^{1/r}(q\wedge\Log m\wedge n)^{(1/p^*-1/r)\vee 0}m^{1/q} ,&2\leq p^*,q
\end{cases}
\end{align*}
\begin{align*}   
&\sim
m^{(1/q-1/2)\vee 0}(p^*\wedge \Log n)^{1/r}(p^*\wedge\Log n\wedge m)^{(1/(q\vee 2)-1/r)\vee 0}n^{1/p^*}
\\
&\qquad
+n^{(1/p^*-1/2)\vee 0}(q\wedge \Log m)^{1/r}(q\wedge\Log m\wedge n)^{(1/(p^*\vee 2)-1/r)\vee 0}m^{1/q}.
\end{align*}
\end{cor}

In particular, when $r=\infty$ we get the following two-sided bound for matrices with iid Rademacher entries $\ve_{i,j}$.

\begin{cor}
\label{cor:Radrect}
If $1\leq p,q\leq\infty$, then
\begin{align*}
\Ex\bigl\|(\ve_{i,j})_{i\leq m,j\leq n}\bigr\|_{\ell_p^n\rightarrow\ell_q^m} 
& \sim
\begin{cases}
m^{1/q-1/2}n^{1/p^*}+n^{1/p^*-1/2}m^{1/q},&p^*,q\leq 2,
\\
 \sqrt{p^*\wedge m \: }m^{1/q-1/2}n^{1/p^*}+m^{1/q},&q\leq 2\leq p^*,
\\
n^{1/p^*}+ \sqrt{q\wedge n\: }n^{1/p^*-1/2}m^{1/q},&p^*\leq 2\leq q,
\\
(p^*\wedge m)^{1/q}n^{1/p^*}+(q\wedge n)^{1/p^*}m^{1/q},&2\leq p^*,q.
\end{cases}
\nonumber 
\\
&\sim
(p^*\wedge m)^{1/(q\vee 2)}m^{(1/q-1/2)\vee 0}n^{1/p^*}
+(q\wedge n)^{1/(p^*\vee 2)}n^{(1/p^*-1/2)\vee 0}m^{1/q}.
\end{align*}
\end{cor}

\begin{rem}	
In \cite[Theorem 11]{LSChevet} we provide two-sided bounds for $\Ex\|(a_ib_jX_{i,j})_{i\leq m,j\leq n}\|_{\ell_p^n\rightarrow\ell_q^m}$, where the vectors $a\in \er^m$  and $b\in \er^n$ are arbitrary, and $X_{i,j}$'s are Weibull random variables with
parameter $r\in [1,2]$. We do not know similar formulas for $r>2$. 
\end{rem}

\subsection{Variables with log-convex tails} 

In this subsection we assume that $X_{i,j}$ have log-convex tails, i.e., the function $N$
given by \eqref{eq:logtail} is concave.

\begin{lem}
\label{lem:momlogconvex}
Let $(X_{i,j})$ be iid symmetric random variables with log-convex tails and assume that 
\eqref{alphareg} holds. Then for  every $p,q\geq 1$,
\[
\sup_{t\in B_p^n}\Bigr\|\sum_{j=1}^nt_jX_{1,j}\Bigl\|_q
\sim_\alpha \|X_{i,j}\|_q 
+ \sqrt{q}\|X_{i,j}\|_2 n^{(1/p^*-1/2)\vee 0}.
\]
\end{lem}

\begin{proof}
If $q\leq 2$, then \eqref{eq:compmomregalpha} 
yields
\begin{align*}
\sup_{t\in B_p^n}\Bigr\|\sum_{j=1}^nt_jX_{1,j}\Bigl\|_q
&\sim_\alpha \sup_{t\in B_p^n}\Bigr\|\sum_{j=1}^nt_jX_{1,j}\Bigl\|_2
=\sup_{t\in B_p^n}\|t\|_2\|X_{i,j}\|_2=n^{(1/p^*-1/2)\vee 0}\|X_{i,j}\|_2
\\
&\sim \|X_{i,j}\|_q 
+ \sqrt{q}\|X_{i,j}\|_2 n^{(1/p^*-1/2)\vee 0}.
\end{align*}

Now assume that $q>2$.
By \cite[Theorem~1.1]{HM-SO} we have
\begin{align*}
\Bigr\|\sum_{j=1}^nt_jX_{1,j}\Bigl\|_q
&\sim
\Bigl(\sum_{j=1}^{n}|t_j|^q\Ex|X_{1,j}|^q\Bigr)^{1/q}+\sqrt{q}\Bigl(\sum_{j=1}^{n}|t_j|^2\Ex|X_{1,j}|^2\Bigr)^{1/2}
\\
&=\|t\|_q\|X_{i,j}\|_q+\sqrt{q}\|t\|_2\|X_{i,j}\|_2
\gtrsim \|t\|_\infty\|X_{i,j}\|_q+\sqrt{q}\|t\|_2\|X_{i,j}\|_2.
\end{align*}
We shall show that the last estimate may be reversed up to a constant depending only on $\alpha$.
To this aim assume put $a:=\|t\|_\infty\|X_{i,j}\|_q+\sqrt{q}\|t\|_2\|X_{i,j}\|_2$. Then
\[
\|t\|_q\|X_{i,j}\|_q\leq (\|t\|_\infty\|X_{i,j}\|_q)^{(q-2)/q}(\|t\|_2\|X_{i,j}\|_q)^{2/q}
\leq a(\|X_{i,j}\|_q/\|X_{i,j}\|_2)^{2/q}\lesssim_\alpha a,
\]
where the last estimate follows by \eqref{eq:compmomregalpha}. Thus, for $q>2$,
\[
\sup_{t\in B_p^n}\Bigr\|\sum_{j=1}^nt_jX_{1,j}\Bigl\|_q
\sim_\alpha
\sup_{t\in B_p^n}(\|t\|_\infty\|X_{i,j}\|_q+\sqrt{q}\|t\|_2\|X_{i,j}\|_2)
\sim \|X_{i,j}\|_q 
+ \sqrt{q}\|X_{i,j}\|_2 n^{(1/p^*-1/2)\vee 0} \qedhere
\]
\end{proof}

\begin{rem}
\label{rem:momlogconvex-big-p*q}
Since $N$ is concave, $N^{-1}$ is convex and  $N^{-1}(0)=0$,  hence 
$N^{-1}(q)\geq \frac q2N^{-1}(2)$ whenever $q\ge 2$. 
So \eqref{alphareg} and  Lemma~\ref{lem:mom_invtail} imply that
$\|X_{i,j}\|_q\sim_\alpha N^{-1}(q)\gtrsim_\alpha q\|X_{i,j}\|_2$. 
Thus, we get by Lemma~\ref{lem:momlogconvex},
\[
\sup_{t\in B_p^n}\Bigr\|\sum_{j=1}^nt_jX_{1,j}\Bigl\|_q
\sim_\alpha \|X_{i,j}\|_q 
\quad \mbox{for } p^*, q\geq 2. 
\]
\end{rem}

Theorem~\ref{thm:rect}, Lemma~\ref{lem:momlogconvex}, and Remark~\ref{rem:momlogconvex-big-p*q} yield the following corollary. 

\begin{cor}
\label{cor:rectlogconvex}
Let $(X_{i,j})_{i\leq m,j\leq n}$ be iid symmetric random variables with log-convex tails such that \eqref{alphareg} holds. Then
\begin{align*}\notag
\Ex&\bigl\|(X_{i,j})_{i\leq m,j\leq n}\bigr\|_{\ell_p^n\rightarrow\ell_q^m}
\\ 
\notag
&\sim_\alpha
\begin{cases}
(m^{1/q-1/2}n^{1/p^*}+n^{1/p^*-1/2}m^{1/q})\|X_{i,j}\|_2,&p^*,q\leq 2,
\\
n^{1/p^*}(m^{1/q-1/2}\sqrt{p^*\wedge\Log n}\|X_{i,j}\|_2+\|X_{i,j}\|_{p^*\wedge\Log n})
+m^{1/q}\|X_{i,j}\|_2,&q\leq 2\leq p^*,
\\
 n^{1/p^*}\|X_{i,j}\|_2+m^{1/q}(n^{1/p^*-1/2}\sqrt{q\wedge\Log m}\|X_{i,j}\|_2+\|X_{i,j}\|_{q\wedge\Log m}),
 &p^*\leq 2\leq q,
\\
n^{1/p^*}\|X_{i,j}\|_{p^*\wedge\Log n}+m^{1/q}\|X_{i,j}\|_{q\wedge\Log m} ,&2\leq p^*,q
\end{cases}
\end{align*}
\end{cor}

\subsubsection{Heavy-tailed Weibull random variables}

Weibull random variables with parameter $r\in (0,1]$ have log-convex tails. Moreover, in this case
$\|X_{i,j}\|_\rho = (\Gamma(1+\rho/r)^{1/\rho}) \sim_r \rho^{1/r}$,
so $X_{i,j}$'s satisfy \eqref{alphareg} with $\alpha\sim 2^{1/r}$ and thus
 Corollary~\ref{cor:square} implies that
\begin{align*}
\Ex\bigl\|(X_{i,j})_{i,j=1}^n\bigr\|_{\ell_p^n\rightarrow\ell_q^n},
&\sim_r 
\begin{cases}
n^{1/q+1/p^*-1/2} &    p^*,q\leq 2,\\
(p^*\wedge q\wedge \Log n)^{1/r}  n^{1/(p^*\wedge q)},  &   p^*\vee q \leq 2
\end{cases} 
\\
&  \sim  (p^*\wedge q\wedge \Log n)^{1/r}n^{1/(p^*\wedge q)}n^{ (1/(p^*\vee q)-1/2)\vee 0}.
\end{align*}

In the rectangular case Corollary \ref{cor:rectlogconvex} yields the following.

\begin{cor}
Let $(X_{i,j})_{i\leq m,j\leq n}$ be iid Weibull random variables with parameter $r\in (0,1]$.
Then for every $1\leq p,q\leq \infty$ we have
\begin{align*}
\Ex\bigl\|(X_{i,j})_{i\leq m,j\leq n} \bigr\|_{\ell_p^n\rightarrow\ell_q^m}
& \sim_r
(q\wedge \Log m)^{1/2} n^{(1/p^*-1/2)\vee 0}m^{1/q} +(q\wedge \Log m)^{1/r}m^{1/q}
\\ & \qquad+
(p^*\wedge \Log n)^{1/2} m^{(1/q-1/2)\vee 0}n^{1/p^*} +(p^*\wedge \Log m)^{1/r}n^{1/p^*}.
\end{align*}

\end{cor}

\subsection{Non-centered random variables} 
\label{sect:noncentered}

In this subsection we prove  \eqref{eq:noncentered-aim} under  centered regularity 
assumption~\eqref{eq:noncentered-assumpt}.  Let $\|\cdot\|$ denote the operator norm 
from $\ell_p^n$ to $\ell_q^m$. Note that
\begin{align*}
\|(\Ex X_{i,j})\|
& = |\Ex X_{1,1}|\cdot \|(1)_{i,j}\|
=  |\Ex X_{1,1}|\cdot\sup_{t\in B_p^n} 
\Bigl(\sum_{i=1}^m \Bigl| \sum_{j=1}^n t_j\Bigr|^q\Bigr)^{1/q} 
=  |\Ex X_{1,1}|\cdot m^{1/q}\sup_{t\in B_p^n} \Bigl| \sum_{j=1}^n t_j\Bigr| 
\\
&  = m^{1/q}n^{1/p^*}|\Ex X_{1,1}|.
\end{align*}
 By the triangle inequality we have
\begin{align*}
\Ex \|(X_{i,j})\|
& \le \Ex \bigl\| (X_{i,j} - \Ex X_{i,j}) \bigr\| + \|(\Ex X_{i,j})\|
=\Ex \bigl\| (X_{i,j} - \Ex X_{i,j}) \bigr\| + m^{1/q}n^{1/p^*}|\Ex X_{1,1}| ,
\end{align*}
so  Theorem~\ref{thm:rect} implies the upper bound in \eqref{eq:noncentered-aim}. Moreover, Jensen's inequality yields $\Ex\|(X_{i,j})\| \ge  \|(\Ex X_{i,j})\|$, so applying triangle inequality we get
\begin{align*}
\Ex \|(X_{i,j})\|
\ge \frac 12\Ex \|(X_{i,j})\|
+ \frac 12 \Bigl( \Ex \bigl\| (X_{i,j} - \Ex X_{i,j}) \bigr\| - \|(\Ex X_{i,j})\|\Bigr) 
 \ge \frac{1}{2} \Ex\bigl\| (X_{i,j} - \Ex X_{i,j}) \bigr\|.
\end{align*}
Hence, Theorem~\ref{thm:rect} and another application of inequality  $\Ex\|(X_{i,j})\| \ge  \|(\Ex X_{i,j})\|=m^{1/q}n^{1/p^*}|\Ex X_{1,1}| $ yield the lower bound in \eqref{eq:noncentered-aim}.


\section{Lower bounds}	
\label{sect:lower-bounds}

In this section we shall prove the 
 lower bound in Theorem \ref{thm:rect}. 
The crucial technical result
we  use is the following lower bound  for $\ell_r$-norms of iid sequences.

\begin{lem}
\label{lem:estellrnorm}
Let $r\geq 1$ and $Y_1,Y_2,\ldots,Y_k$ be iid nonnegative random variables satisfying the
condition $\|Y_i\|_{2r}\leq \alpha\|Y_i\|_r$ for some $\alpha\in [1,\infty)$.
Assume that $k\ge 4\alpha^{2r}$.
Then
\[
\Ex\Bigl(\sum_{i=1}^k Y_i^r\Bigr)^{1/r}\geq \frac{1}{128\alpha^2} k^{1/r}\|Y_1\|_r.
\]
\end{lem}

\begin{proof}
Define
\[
Z:=\sum_{i=1}^k1_{A_i},\quad A_i:=\Bigl\{Y_i^r\geq \frac{1}{2}\Ex Y_i^r\Bigr\}. 
\]
The Paley-Zygmund inequality yields
\[
\Pr(A_i)\geq \frac{1}{4}\frac{(\Ex Y_i^r)^2}{\Ex Y_i^{2r}}\geq \frac{1}{4}\alpha^{-2r}.
\]
Since $k\geq 4\alpha^{2r}$, this gives
\[
\Ex Z=\sum_{i=1}^k\Pr(A_i)\geq \frac{k}{4}\alpha^{-2r}\geq 1
\]
and
\[
\Ex Z^2=2\sum_{1\leq i<j\leq k}\Pr(A_i)\Pr(A_j)+\sum_{i=1}^k\Pr(A_i)\leq (\Ex Z)^2+\Ex Z\leq
2(\Ex Z)^2.
\]
Applying again the Paley-Zygmund inequality we obtain
\[
\Pr\Bigl(Z\geq \frac{1}{2}\Ex Z\Bigr)\geq \frac{1}{4}\frac{(\Ex Z)^2}{\Ex Z^2}\geq \frac{1}{8}.
\]
Hence, 
\[
\Ex\Bigl(\sum_{i=1}^k Y_i^r\Bigr)^{1/r}
\geq \Pr\Bigl(Z\geq \frac{1}{2}\Ex Z\Bigr)\Bigl(\frac{1}{2}\Ex Z\frac{1}{2}\Ex Y_i^r\Bigr)^{1/r}
\geq \frac{1}{8}\Bigl(\frac{k}{16}\alpha^{-2r}\Ex Y_i^r\Bigr)^{1/r}
\geq \frac{1}{128\alpha^2}k^{1/r}\|Y_i\|_r.		\qedhere
\]
\end{proof}

\begin{proof}[Proof of the lower bound in Theorem~\ref{thm:rect}]
Let us fix $t\in B_p^n$ and put $Y_i:=|\sum_{j=1}^nt_jX_{i,j}|$. Then $Y_1,\ldots,Y_m$
are iid random variables. Moreover, by \eqref{eq:compmomregalpha}, 
$\|Y_i\|_{2r}\leq \tilde{\alpha}\|Y_i\|_r$ for $r\geq 1$, where a~constant 
$\tilde{\alpha}\geq 1$ depends only on $\alpha$.

If $m\geq 4\tilde{\alpha}^{2q}$, then by Lemma \ref{lem:estellrnorm} we get
\[
\Ex\bigl\|(X_{i,j})_{i\leq m, j\leq n}\bigr\|_{\ell_p^n\rightarrow\ell_q^m}
\geq \Ex\Bigl(\sum_{i=1}^m Y_i^q\Bigr)^{1/q}\geq \frac{1}{128\tilde{\alpha}^2}m^{1/q}\|Y_i\|_q.
\]

If $m\leq 4\tilde{\alpha}^{4}$, then by \eqref{eq:compmomregalpha} we have
\[
\Ex\bigl\|(X_{i,j})_{i\leq m, j\leq n}\bigr\|_{\ell_p^n\rightarrow\ell_q^m}
\geq \|Y_i\|_1\gtrsim_\alpha \|Y_i\|_{\Log m}\sim_\alpha m^{1/q}\|Y_i\|_{ q\wedge \Log m}.
\]

If $4\tilde{\alpha}^{4}\leq m\leq 4\tilde{\alpha}^{2q}$, then
$m=4 \tilde{\alpha}^{2\tilde{q}}$ for some $1\leq \tilde{q}\leq q$.
Moreover, in this case $m^{1/q}\sim_\alpha 1\sim_\alpha m^{1/\tilde{q}}$
and $\tilde{q}\sim_\alpha  q\wedge \Log m$. Hence, Lemma \ref{lem:estellrnorm}  and  \eqref{eq:compmomregalpha} yield
\begin{align*}
\Ex\bigl\|(X_{i,j})_{i\leq m, j\leq n}\bigr\|_{\ell_p^n\rightarrow\ell_q^m}
& \geq \Ex\Bigl(\sum_{i=1}^m Y_i^q\Bigr)^{1/q}\sim_\alpha 
\Ex\Bigl(\sum_{i=1}^m Y_i^{\tilde{q}}\Bigr)^{1/\tilde{q}}
\\
& \geq  \frac{1}{128\tilde{\alpha}^2}m^{1/\tilde{q}}\|Y_i\|_{\tilde{q}}
\sim_\alpha m^{1/q}\|Y_i\|_{ q\wedge \Log m}.
\end{align*}

The argument above shows that
\[
\Ex\bigl\|(X_{i,j})_{i\leq m, j\leq n}\bigr\|_{\ell_p^n\rightarrow\ell_q^m}
\gtrsim_{\alpha}
m^{1/q}\sup_{t\in B_p^n}\Bigl\|\sum_{j=1}^nt_jX_{1,j}\Bigr\|_{ q\wedge \Log m}.
\]
The bound by the other term follows  by the following duality
\begin{equation}	\label{eq:duality}
\bigl\|(X_{i,j})_{i\leq m, j\leq n}\bigr\|_{\ell_p^n\rightarrow\ell_q^m}
=\bigl\|(X_{j,i})_{j\leq n, i\leq m}\bigr\|_{\ell_{q^*}^m\rightarrow\ell_{p^*}^n}.\qedhere
\end{equation}
\end{proof}


\section{Formula in the square case }		
\label{sect:largergeq2}

This section contains  proofs of Propositions~\ref{prop:squarelargep*} and  \ref{rem:smallq}, which  immediately yield the equivalence of formulas from Theorem~\ref{thm:rect} and Corollary~\ref{cor:square} in the square case.

\begin{proof}[Proof of Proposition~\ref{prop:squarelargep*}]
 By duality it suffices to show that for $p^*\geq q\vee 2$,
\begin{equation}	\label{eq:est-aim-rem}
n^{1/q}\sup_{t\in B_p^n}\Bigl\|\sum_{j=1}^nt_jX_{1,j}\Bigr\|_{ q\wedge \Log n}
+n^{1/p^*}\sup_{s\in B_{q^*}^n}\Bigl\|\sum_{i=1}^n s_iX_{i,1}\Bigr\|_{ p^*\wedge \Log n}
\sim_\alpha n^{1/q}\|X_{1,1}\|_{ q\wedge \Log n}.
\end{equation}

The lower bound is obvious (with constant 1).
 To derive the upper bound we observe first that if we substituted
 $q$ and $p^*$ by $ q\wedge \Log n$ and 
$ p^*\wedge \Log n$, respectively, then the RHS
of \eqref{eq:est-aim-rem} would increase, whereas the LHS would increase only by a constant factor.
So it is enough to consider the case $\Log n\geq p^*\geq  q\vee 2$.

Now we shall show that
\begin{equation}
\label{eq:est1rem}
\Bigl\|\sum_{j=1}^nt_jX_{1,j}\Bigr\|_q \lesssim_\alpha  \|X_{1,1}\|_q \quad 
\mbox{ for every }t\in B_p^n. 
\end{equation}
 To this end fix $t\in B_p^n$ and assume without loss of generality that $t_1\geq t_2\geq\cdots\geq t_n\geq 0$.
If $1\leq q\leq 4$, then by \eqref{eq:compmomregalpha} we have
\[
\Bigl\|\sum_{j=1}^nt_jX_{1,j}\Bigr\|_q\lesssim_{\alpha} \Bigl\|\sum_{j=1}^nt_jX_{1,j}\Bigr\|_2
=\|t\|_2\|X_{1,1}\|_2 \le \alpha \|X_{1,1}\|_1 \le \alpha \|X_{1,1,}\|_q.
\]
If $q\geq 4$, then
\[
\Bigl\|\sum_{j\leq e^{4q}}t_jX_{1,j}\Bigr\|_q
\leq \sum_{j\leq e^{4q}}|t_j|\|X_{1,1}\|_q
\leq e^{4q/p^*}\|t\|_p\|X_{1,1}\|_q\leq e^4\|X_{1,1}\|_q.
\]
Moreover, by Rosenthal's inequality \cite[Theorem 1.5.11]{dPG}, 
\[
\Bigl\|\sum_{j> e^{4q}}t_jX_{1,j}\Bigr\|_q
\leq 
C\frac{q}{\Log q}(\|(t_j)_{j> e^{4q}}\|_2\|X_{1,1}\|_2+\|(t_j)_{j> e^{4q}}\|_q\|X_{1,1}\|_q).
\]
If  $j>e^{4q}$, then $t_j\leq  j^{-1/p}\leq e^{- 4q/p}$,  so for $p^*\ge q\ge 4$ we have
\[
\|(t_j)_{j> e^{4q}}\|_q
\leq
\|(t_j)_{j> e^{4q}}\|_2 
\leq \|t\|_p^{p/2} \max_{j>e^{4q}} t_j^{(2-p)/2}
\leq \|t\|_p^{p/2}(e^{- 4q/p})^{1-p/2}
\leq e^{- q}
\]
and \eqref{eq:est1rem} follows.

 To conlude the proof it is enough to show that for $\Log n\geq p^*\geq  q\vee 2$,
\begin{equation}
\label{eq:est2rem}
n^{1/p^*}\sup_{s\in B_{q^*}^n}\Bigl\|\sum_{i=1}^n s_iX_{i,1}\Bigr\|_{p^*} 
\lesssim_\alpha n^{1/q}\sup_{t\in B_p^n}\Bigl\|\sum_{j=1}^nt_jX_{1,j}\Bigr\|_{ q} 
+ n^{1/q}\|X_{1,1}\|_q.
\end{equation}

For $k=0,1,\ldots$ define $\rho_k:=32\beta^2\Log^{(k)}(p^*),$ where 
 $\Log^{(k+1)}x\coloneqq \Log(\Log^{(k)}x)$, $\Log^{(0)}x\coloneqq x$, 
 and $\beta=\frac 12\vee \log_2\alpha$.
Observe that $(\rho_k)_k$ is non-increasing and for large $k$ we have $\rho_k=32\beta^2$.

If $p^*/q\leq 32\beta^2$, i.e., $p^*\leq 32\beta^2q$, then \eqref{eq:compmomregalpha} implies that 
\[
\Bigl\|\sum_{i=1}^n s_iX_{i,1}\Bigr\|_{p^*}
\lesssim_\alpha \Bigl(\frac{p^*}{q}\Bigr)^\beta\Bigl\|\sum_{i=1}^n s_iX_{i,1}\Bigr\|_{q}
\lesssim_\alpha \Bigl\|\sum_{i=1}^n s_iX_{i,1}\Bigr\|_{q}.
\] 
Moreover, $B_{q^*}^n \subset n^{1/q-1/p^*}B_p^n$, so in this case
\[
n^{1/p^*}\sup_{s\in B_{q^*}^n}\Bigl\|\sum_{i=1}^n s_iX_{i,1}\Bigr\|_{p^*}
\lesssim_\alpha n^{1/q}\sup_{t\in B_p^n}\Bigl\|\sum_{j=1}^nt_jX_{1,j}\Bigr\|_{q}
\]
and \eqref{eq:est2rem} follows.

Now suppose that $\rho_{k}< p^*/q\leq \rho_{k-1}$ for some $k\geq 1$. Define
$q_k:=2p^*/\rho_k \ge q\vee 2$.  Estimates \eqref{eq:compmomregalpha} and \eqref{eq:est1rem}, 
applied with  $p^* \coloneqq q_k$ and $q \coloneqq q_k \ge 2$,  yield
\[
\sup_{s\in B_{q_k^*}}\Bigl\|\sum_{i=1}^n s_iX_{i,1}\Bigr\|_{p^*}
\lesssim_\alpha \rho_k^\beta \sup_{s\in B_{q_k^*}}\Bigl\|\sum_{i=1}^n s_iX_{i,1}\Bigr\|_{q_k}
\lesssim_\alpha \rho_k^\beta\|X_{1,1}\|_{q_k}
\lesssim_\alpha \Bigl(\frac{\rho_k q_k}{q}\Bigr)^\beta\|X_{1,1}\|_q.
\]
Since $q_k\geq q$ we have $B_{q^*}^n\subset n^{1/q-1/q_k} B_{q_k^*}^n$.
Therefore,
\[
n^{1/p^*}\sup_{s\in B_{q^*}^n}\Bigl\|\sum_{i=1}^n s_iX_{i,1}\Bigr\|_{p^*} 
\lesssim_\alpha \Bigl(\frac{p^*}{q}\Bigr)^\beta n^{1/p^*-1/q_k} n^{1/q}\|X_{1,1}\|_q 
=  \Bigl(\frac{p^*}{q}\Bigr)^\beta n^{\frac{2-\rho_k}{2p^*}} n^{1/q}\|X_{1,1}\|_q .
\]
Hence, it is enough to show that
\begin{equation}
\label{eq:toshow}
\Bigl(\frac{p^*}{q}\Bigr)^\beta n^{\frac{2-\rho_k}{2p^*}}\leq 1.
\end{equation}
Observe that $p^*/q\geq 32\beta^2\geq 8$, so $\Log n\geq p^*\geq 8q\geq 8$, 
$\Log(p^*/q)=\ln(p^*/q)$, and $\Log n=\ln n$. Thus,
\eqref{eq:toshow} is equivalent to
\begin{equation}
\label{eq:toshow2}
\frac{\rho_k-2}{2\beta \Log(\frac{p^*}{q})}\geq \frac{p^*}{\Log n}.
\end{equation}
We have $p^*/\Log n\leq 1$ and
\[
\frac{\rho_k-2}{2\beta \Log(\frac{p^*}{q})}
\geq \frac{24\beta^2\Log^{(k)}(p^*)}{2\beta \Log\rho_{k-1}}
\geq \frac{24\beta^2-2\beta+2\beta\Log^{(k)}(p^*)}{2\beta \ln(32\beta^2) +2\beta\Log^{(k)}(p^*)}
\geq 1,
\]
where in the first inequality we used $\Log^{(k)}x\geq 1$ and $8\beta^2\geq 2$,
in the second one $\Log(ab)\leq \ln a+\Log b$ for $a\geq 1$, and in the last one
$\ln(32e\beta^2)\leq 12\beta$ for $\beta \geq 1/2$.
\end{proof}

Now we move to the proof of Proposition~\ref{rem:smallq}.  Observe that $m,n$ are arbitrary (not necessarily $m=n$).

\begin{proof}[Proof of Proposition~\ref{rem:smallq}]
It is enough to establish the first part of the assertion.
We have
\[
\sup_{t\in B_p^n}\Bigl\|\sum_{j=1}^nt_jX_{j}\Bigr\|_{\tilde{q}}
\leq \sup_{t\in B_p^n}\Bigl\|\sum_{j=1}^nt_jX_{j}\Bigr\|_{2}
=\sup_{t\in B_p^n}\|t\|_2\|X_{1}\|_2
\]
and the upper bound immediately follows.

If $p\leq 2$ then $(1/p^*-1/2)_+=0$ and the lower bound is obvious 
(with constant $1$ instead of $1/2\sqrt{2}$). Assume that $p>2$. 
Let  $(X_j')_j$ be an independent copy of $(X_j)_j$, and let $\ve_i$'s be iid Rademachers independent of other variables. 
Then		
\begin{align*}
\sup_{t\in B_p^n}\Bigl\|\sum_{j=1}^nt_jX_{j}\Bigr\|_{\tilde{q}}
&\geq n^{-1/p}\Bigl\|\sum_{j=1}^nX_{j}\Bigr\|_{\tilde{q}}
\geq \frac 12 n^{-1/p}\Bigl\|\sum_{j=1}^n(X_{j}-X_j')\Bigr\|_{\tilde{q}}
=  \frac 12 n^{-1/p}\Bigl\|\sum_{j=1}^n\ve_j(X_{j}-X_j')\Bigr\|_{\tilde{q}}
\\& \ge \frac 12 n^{-1/p}\Bigl\|\sum_{j=1}^n\ve_j(X_{j}-\Ex X_j')\Bigr\|_{\tilde{q}}
= \frac{1}{2}n^{-1/p}\Bigl\|\sum_{j=1}^n\ve_jX_{j}\Bigr\|_{\tilde{q}}.
\end{align*}
Moreover,
Khintchine's and H\"older's inequalities yield (recall that $\tilde{q}\in [1,2]$)
\[
\Ex\Bigl|\sum_{j=1}^n\ve_jX_{j}\Bigr|^{\tilde{q}}
\geq 2^{-\tilde{q}/2}\Ex\Bigl(\sum_{j=1}^nX_{j}^2\Bigr)^{\tilde{q}/2}
\geq 2^{-\tilde{q}/2}n^{\tilde{q}/2-1}\Ex\sum_{j=1}^n|X_{j}|^{\tilde{q}}
=2^{-\tilde{q}/2}n^{\tilde{q}/2}\Ex|X_{1}|^{\tilde{q}}. \qedhere
\]
\end{proof}

\section{Upper bounds}	
\label{sect:upper-bounds}

 To prove the upper bound in Theorem 1 we split the range $p^*,q\geq 1$ into several parts.
 In each of them we use different arguments to derive the asserted estimate.

\subsection{Case $p^*,q\leq 2$}	
\label{seq:case-leq2}

In this subsection we shall show that the two-sided bound from Theorem \ref{thm:rect}  holds in the range
$p^*,q\leq 2$ under the following mild 4th moment assumption
\begin{equation}
\label{eq:4mom}
(\Ex X_{1,1}^4)^{1/4}\leq \alpha (\Ex X_{1,1}^2)^{1/2}.
\end{equation}

Observe that then H\"older's inequality yields
\[
\Ex X_{1,1}^2\leq (\Ex X_{1,1}^4)^{1/3}(\Ex |X_{1,1}|)^{2/3}
\leq \alpha^{4/3}(\Ex X_{1,1}^2)^{2/3}(\Ex |X_{1,1}|)^{2/3},
\]
so
\begin{equation}
\label{eq:1_2mom}
\Ex|X_{1,1}|\geq \alpha^{-2}(\Ex X_{1,1}^2)^{1/2}.
\end{equation}

Let us first consider the case $p=q=2$. 
Then we shall see that it may be easily extrapolated into the whole range of $p^*,q\le 2$.

\begin{prop}
\label{prop:twosided22}
Let $(X_{i,j})_{ i\leq m ,j\leq n}$ be iid centered random variables satisfying \eqref{eq:4mom}. 
Then
\[
\Ex\bigl\|(X_{i,j})_{i\leq m,j\leq n}\bigr\|_{\ell_2^n\rightarrow\ell_2^m}
\sim_{\alpha}(\Ex X_{1,1}^2)^{1/2}(\sqrt{n}+\sqrt{m}).
\]
\end{prop} 

\begin{proof}
By \cite[Theorem 2]{La} we have
\begin{align*}
\Ex\bigl\|(X_{i,j})_{i\leq m,j\leq n}\bigr\|_{\ell_2^n\rightarrow\ell_2^m}
&\lesssim\max_j\sqrt{\sum_{i}\Ex X_{i,j}^2}+\max_i\sqrt{\sum_{j}\Ex X_{i,j}^2}+
\sqrt[4]{\sum_{i,j}\Ex X_{i,j}^4}
\\
&\leq (\Ex X_{1,1}^2)^{1/2}(\sqrt{n}+\sqrt{m}+\alpha\sqrt[4]{nm})
\lesssim_{\alpha} (\Ex X_{1,1}^2)^{1/2}(\sqrt{n}+\sqrt{m}).
\end{align*} 
To get the lower bound we use  Jensen's inequality and \eqref{eq:1_2mom}:
\begin{align*}
\Ex\bigl\|(X_{i,j})_{i\leq m,j\leq n}\bigr\|_{\ell_2^n\rightarrow\ell_2^m}
&\geq 
\max\Bigl\{\Ex\bigl\|(|X_{i,1}|)_{i\leq m}\bigr\|_2,\Ex\bigl\|(|X_{1,j}|)_{j\leq n}\bigr\|_2 \Bigr\}
\\
&\geq \max\Bigl\{\bigl\|(\Ex|X_{i,1}|)_{i\leq m}\bigr\|_2,\bigl\|(\Ex|X_{1,j}|)_{j\leq n}\bigr\|_2 \Bigr\}
\geq \alpha^{-2}(\Ex X_{1,1}^2)^{1/2}\sqrt{n\vee m}. \qedhere
\end{align*}

\end{proof}

\begin{cor}\label{cor-p*,q-less2}
Let $(X_{i,j})_{ i\leq m,j\leq n}$ be iid centered random variables  satisfying \eqref{eq:4mom}. 
Then for $p^*,q\le 2$ we have
\[
\Ex\bigl\|(X_{i,j})_{i\leq m,j\leq n}\bigr\|_{\ell_p^n\rightarrow\ell_q^m}
\sim_{\alpha}(\Ex X_{1,1}^2)^{1/2}(m^{1/q-1/2}n^{1/p^*}+n^{1/p^*-1/2}m^{1/q}).
\]
\end{cor} 

\begin{proof}
 Let $(\ve_{i,j})_{i,j}$ be iid symmetric $\pm 1$ random variables independent of $(X_{i,j})$. Symmetrization (as in the proof of Proposition~\ref{rem:smallq}) and \eqref{eq:1_2mom} yields
\begin{align*}
\Ex\bigl\|(X_{i,j})_{i\leq m,j\leq n}^n\bigr\|_{\ell_p^n\rightarrow\ell_q^m}
&\geq \frac{1}{2}\Ex\bigl\|(\ve_{i,j}|X_{i,j}|)_{i\leq m,j\leq n}^n\bigr\|_{\ell_p^n\rightarrow\ell_q^m}
\geq \frac{1}{2}\Ex\bigl\|(\ve_{i,j}\Ex|X_{i,j}|)_{i\leq m,j\leq n}\bigr\|_{\ell_p^n\rightarrow\ell_q^m}
\\
&\gtrsim_\alpha (\Ex X_{1,1}^2)^{1/2}\Ex\bigl\|(\ve_{i,j})_{i\leq m,j\leq n}^n\bigr\|.
\end{align*}
We have
\begin{align*}
\Ex\bigl\|(\ve_{i,j})_{i\leq m,j\leq n}^n\bigr\|_{\ell_p^n\rightarrow\ell_q^m}
&\geq
n^{-1/p}\Ex\Bigl\|\Bigl(\sum_{j=1}^n\ve_{i,j}\Bigr)_{i\leq m}\Bigr\|_q
\sim n^{1/p^*-1}\Bigl(\Ex\Bigl\|\Bigl(\sum_{j=1}^n\ve_{i,j}\Bigr)_{i\leq m}\Bigr\|_q^q\Bigr)^{1/q}
\\
&=n^{1/p^*-1}m^{1/q}\Bigl\|\sum_{j=1}^n\ve_{1,j}\Bigr\|_q
\sim
n^{1/p^*-1/2}m^{1/q},
\end{align*}
where in the first line we used the Kahane-Khintchine and in the second one the Khintchine inequalities. By duality \eqref{eq:duality} we get
\begin{align*}
\Ex\bigl\|(\ve_{i,j})_{i\leq m,j\leq n}\bigr\|_{\ell_p^n\rightarrow\ell_q^m}
=\Ex\bigl\|(\ve_{i,j})_{i\leq n,j\leq m}\bigr\|_{\ell_{q^*}^m\rightarrow\ell_{p^*}^n}
\gtrsim m^{1/q-1/2}n^{1/p^*},
\end{align*}
so the lower bound follows.

To get the upper bound we use Proposition \ref{prop:twosided22} together with the following  simple bound
\begin{align*}
\bigl\|(X_{i,j})_{i\leq m,j\leq n}\bigr\|_{\ell_p^n\rightarrow\ell_q^m}
&\leq 
\|\mathrm{Id}\|_{\ell_p^n\rightarrow\ell_2^n}\bigl\|(X_{i,j})_{i\leq m,j\leq n}\bigr\|_{\ell_2^n\rightarrow\ell_2^m}\|\mathrm{Id}\|_{\ell_2^m\rightarrow\ell_q^m}
\\
&=n^{1/2-1/p}m^{1/q-1/2}\bigl\|(X_{i,j})_{i\leq m,j\leq n}\bigr\|_{\ell_2^n\rightarrow\ell_2^m}.	\qedhere
\end{align*}
\end{proof}

Corollary~\ref{cor-p*,q-less2}, Proposition~\ref{rem:smallq}
and \eqref{eq:1_2mom} yield that under condition~\eqref{eq:4mom} Theorem~\ref{thm:rect} holds 
whenever $p^*,q\le 2$. 
Moreover, one may prove by repeating the same arguments that  the  two-sided estimate
\[
\Ex\bigl\|(X_{i,j})_{i\leq m, j\leq n}\bigr\|_{\ell_p^n\rightarrow\ell_q^m}
 \sim_\alpha
 m^{1/q-1/2}n^{1/p^*}+n^{1/p^*-1/2}m^{1/q}
\]
holds for every $p^*,q\le 2$ and independent random variables $X_{i,j}$ satisfying \eqref{eq:4mom} and $\Ex X_{i,j}^2=1$ (we do not need to assume that $X_{i,j}$'s are identically distributed).


\subsection{Case $p^*\geq \Log n$ or $q\geq \Log m$}	
\label{seq:case-leqLog}

In this subsection we shall show that Theorem \ref{thm:rect} holds under the regularity assumption
\eqref{alphareg} if $p^*\geq \Log n$ or $q\geq \Log m$.

\begin{rem}
\label{rem:aboveLog}
For $p^*\geq \Log n$,  $\tilde{q}\in[1,\infty)$ and iid random variables $X_i$ we have
\[
\|X_1\|_{\tilde{q}}\leq
\sup_{t\in B_p^n}\Bigl\|\sum_{j=1}^nt_jX_{j}\Bigr\|_{\tilde{q}}
\leq e\|X_1\|_{\tilde{q}}.
\]
Similarly, for  $q\geq \Log m$ and  $\tilde{p}\in[1,\infty)$,
\[
\|X_1\|_{\tilde{p}}\leq
\sup_{t\in B_{q^*}^m}\Bigl\|\sum_{j=1}^nt_jX_{j}\Bigr\|_{\tilde{p}}
\leq e\|X_1\|_{\tilde{p}}. 
\]
\end{rem}

\begin{proof}
The lower bounds are obvious. To see the first upper bound it is enough to use the triangle inequality in $L_{\tilde{q}}$ and observe that 
$\|t\|_1\leq n^{1/p^*}\|t\|_p\leq e $ for $p^*\geq \Log n$ and $\ t\in B_p^n$.
\end{proof}

 By Remark \ref{rem:aboveLog}, Theorem~\ref{thm:rect} in the case $p^*\geq \Log n$ or $q\geq \Log m$ reduces to the following
statement.

\begin{prop} 	
\label{prop:above-log}
Let $(X_{i,j})_{i\leq n,j\leq n}$ be iid centered random variables such that \eqref{alphareg} holds.
Then for $q\geq \Log m$,
\[
\Ex\bigl\|(X_{i,j})_{i\leq m,j\leq n}\bigr\|_{\ell_p^n\rightarrow\ell_q^m}
 \sim_\alpha
\sup_{t\in B_p^n}\Bigl\|\sum_{j\leq n}t_jX_{1,j}\Bigr\|_{\Log m}
+n^{1/p^*}\|X_{1,1}\|_{ p^*\wedge \Log n}.
\]
Analogously, for $p^*\geq \Log n$,
\[
\Ex\bigl\|(X_{i,j})_{i\leq m,j\leq n}\bigr\|_{\ell_p^n\rightarrow\ell_q^m}
 \sim_\alpha  
\sup_{s\in B_{q^*}^m}\Bigl\|\sum_{i\leq m}s_iX_{i,1}\Bigr\|_{\Log n}
+m^{1/q}\|X_{1,1}\|_{ q\wedge \Log m}.
\]
\end{prop}

\begin{proof} 
The lower bounds follow by Section~\ref{sect:lower-bounds} and Remark~\ref{rem:aboveLog}. 
Hence, we should establish only the upper bounds.

By duality \eqref{eq:duality} it is enough to consider the case $q\geq \Log m$. 
We have $\|(x_i)_{i\leq m}\|_\infty \leq\|(x_i)_{i\leq m}\|_q\leq e\|(x_i)_{i\leq m}\|_\infty$, so
\[
\bigl\|(X_{i,j})_{i\leq m,j\leq n}\bigr\|_{\ell_p^n\rightarrow\ell_q^m}
\sim \max_{i\leq m}\bigl\|(X_{i,j})_{j\leq n}\bigr\|_{p^*}.
\]

Note that for arbitrary random variables $Y_1,\ldots,Y_k$ we have
\begin{equation}
\label{eq:maxk}
\Ex\max_{i\leq k}|Y_i|\leq \bigl\|\max_{i\leq k}|Y_i|\bigr\|_{\Log k}
\leq \Bigl(\sum_{i\leq k}\Ex|Y_i|^{\Log k}\Bigr)^{1/\Log k}\leq  e\max_{i\leq k}\|Y_i\|_{\Log k},
\end{equation}
Hence,
\[
\Ex\bigl\|(X_{i,j})_{i\leq m,j\leq n}\bigr\|_{\ell_p^n\rightarrow\ell_q^m}
\lesssim \bigl\|\bigl\|(X_{1,j})_{j\leq n}\bigr\|_{p^*}\bigr\|_{\Log m}.
\]

Inequality \eqref{eq:comp-moments-regular} (applied with $m=1$, $U=\{1\}\otimes B_p^n$, and $\rho = \Log m$) implies
\[
\Bigl\|\bigl\|(X_{1,j})_{j\leq n}\bigr\|_{p^*}\Bigr\|_{\Log m}
 \sim_\alpha
\Ex\bigl\|(X_{1,j})_{j\leq n}\bigr\|_{p^*}
+\sup_{t\in B_p^n}\Bigl\|\sum_{j\leq n}t_jX_{1,j}\Bigr\|_{\Log m}.
\]
If $p^*\geq \Log n$ then
\[
\Ex\bigl\|(X_{1,j})_{j\leq n}\bigr\|_{p^*}
\sim \Ex\max_{j\leq n}|X_{1,j}|\lesssim \|X_{1,1}\|_{\Log n},
\]
 where the last bound follows by \eqref{eq:maxk}.
In the case $p^*\leq \Log n$ we have 
\[
\Ex\bigl\|(X_{1,j})_{j\leq n}\bigr\|_{p^*}
\leq \Bigl(\Ex\bigl\|(X_{1,j})_{j\leq n}\bigr\|_{p^*}^{p^*}\Bigr)^{1/p^*}
=n^{1/p^{*}}\|X_{1,1}\|_{p^*}. \qedhere
\]
\end{proof}


\subsection{Outline of proofs of upper bounds in remaining ranges}
\label{sect:outline}

Let us first note that we may 
 assume that random variables $X_{i,j}$ are symmetric, due to the following remark.

\begin{rem}
\label{rem:symmetrization}
It suffices to prove the upper bound from Theorem~\ref{thm:rect} under additional assumption that random variables $X_{ij}$ 
are symmetric. 
\end{rem}
 
\begin{proof} Let $(X_{i,j}')_{i\le m, j\le n}$ be an independent copy of a random matrix $(X_{i,j})_{i\le m, j\le n}$,
and let $Y_{i,j}=X_{i,j}-X_{i,j}'$. Then \eqref{alphareg} implies for every $\rho\ge 1$,
\begin{align*}
\|Y_{i,j}\|_{2\rho} 
&\leq \|X_{i,j}\|_{2\rho} + \|X_{i,j}'\|_{2\rho}
=2 \|X_{i,j}\|_{2\rho}
\le 2\alpha \|X_{i,j}\|_\rho 
=2\alpha \|X_{i,j} - \Ex X_{i,j}' \|_\rho 
\\
&\leq 2\alpha \|X_{i,j} -  X_{i,j}' \|_\rho 
= 2\alpha \|Y_{i,j}\|_{\rho} .
\end{align*}
Therefore, $(Y_{i,j})_{i\le m, j \le n}$ are iid symmetric random variables satisfying \eqref{alphareg} with $\alpha\coloneqq 2\alpha$.
Moreover, 
\begin{align*}
\Ex\sup_{s\in S, t \in T} \sum_{i\le m, j\le n} X_{i,j}s_it_j 
&= \Ex\sup_{s\in S, t \in T} \sum_{i\le m, j\le n} (X_{i,j} - \Ex X_{i,j}')s_it_j 
\\ 
& 
\le \Ex\sup_{s\in S, t \in T} \sum_{i\le m, j\le n} (X_{i,j} -  X_{i,j}')s_it_j 
= \Ex\sup_{s\in S, t \in T} \sum_{i\le m, j\le n} Y_{i,j}s_it_j,
\end{align*}
so it suffices to upper bound $\Ex\sup_{s\in S, t \in T} \sum_{i\le m, j\le n} Y_{i,j}s_it_j$ by 
\begin{multline*}
m^{1/q}\sup_{t\in B_p^n}\Bigl\|\sum_{j=1}^nt_jY_{1,j}\Bigr\|_{ q\wedge \Log m}
+n^{1/p^*}\sup_{s\in B_{q^*}^m}\Bigl\|\sum_{i=1}^{m} s_iY_{i,1}\Bigr\|_{ p^*\wedge \Log n}
\\ \le 2 m^{1/q}\sup_{t\in B_p^n}\Bigl\|\sum_{j=1}^nt_jX_{1,j}\Bigr\|_{ q\wedge \Log m}
+2n^{1/p^*}\sup_{s\in B_{q^*}^m}\Bigl\|\sum_{i=1}^{m} s_iX_{i,1}\Bigr\|_{ p^*\wedge \Log n}.
\qedhere
\end{multline*}
\end{proof}
 
 We shall also assume without loss of generality that $\alpha\geq \sqrt{2}$. Then \eqref{eq:compmomregalpha} holds
with $\beta=\log_2\alpha$.

One of the ideas used in the sequel is to decompose
certain  subsets  $S$ of $B_{q^*}^m$  and $T$ of $B_p^n$ in the following way. 
Let $T$ be a monotone subset of $B_p^n$ (we  need the monotonicity only to guarantee that if 
$t\in T$ and $I\subset [n]$, then $(tI_{\{i\in I\}})\in T$). Fix $a\in (0,1]$ and 
write $t\in T$ as  $t=(t_{i}I_{\{|t_i|\leq a\}})+(t_{i}I_{\{|t_i|> a\}})$.
Since $a^p\, |\{i\colon |t_i|> a\}|\leq \|t\|_p\leq 1$, we get $T\subset T_1+T_2$, 
where
\[
T_1=T\cap aB_\infty^n, \qquad T_2=\{t\in T\colon |\supp t|\le a^{-p}\}.
\] 
Choosing $a=k^{-1/p}$ we see that for every $1\leq k\leq n$ we have $T\subset T_1+T_2$, where
\[
T_1=T\cap k^{-1/p}B_\infty^n, \qquad T_2=\{t\in T\colon |\supp t|\le k\}.
\]
Similarly, we may also decompose monotone subsets $S$ of $B_{q^*}^m$ into two parts: one containing vectors with bounded $\ell_\infty$-norm and the other containing vectors with bounded support.

Once we decompose $B_p^n$ and $B_{q^*}^m$ as above,  we need to control the quantities of the form 
$\Ex \sup_{s\in S, t\in T} \sum X_{i,j}s_it_j$ provided we have additional information about the 
$\ell_\infty$-norm or the size of the support (or both of them) for vectors from $S$ and $T$.  
 In the next subsection  we present a~couple of lemmas allowing to upper bound this type of quantities 
in various situations.

\subsection{Tools used in proofs of upper bounds in remaining ranges} 
\label{sect:tools}

\begin{lem} 	
\label{lem:S-bdd-T-bddsupp}
Assume that $k,l\in\zet_{+}$, $p^*,q\geq 1$, $a,b>0$ and $(X_{i,j})_{i\leq m,j\leq n}$ are iid symmetric random variables satisfying \eqref{alphareg} with $\alpha\ge \sqrt{2}$,  and $\Ex X_{i,j}^2=1$. 
Denote $\gamman= \log_{2}\alpha$.
	
If $q\geq 2$, $S\subset B_{q^*}^m\cap aB_{\infty}^m$ and 
$T\subset \{t\in B_{p}^n \colon |\supp(t)|\le k\}$, 
then
\begin{equation}	
\label{eq:lem-S-bdd-T-bddsupp}
\Ex\sup_{s\in S, t \in T} \sum_{i\le m, j\le n} X_{i,j}s_it_j 
\lesssim_\alpha m^{1/q} \sup_{t\in T} \Bigl\|\sum_{j=1}^n X_{1,j}t_j \Bigr\|_q 
+ \bigl( n\wedge(k\Log n)\bigr)^{\gamman}k^{(1/p^*-1/2)\vee 0} a^{(2-q^*)/2}.
\end{equation}
	
If $p^*\geq 2$, $S\subset \{s\in B_{q^*}^m \colon |\supp(s)|\le l\}$ and 
$T\subset B_{p}^n\cap bB_{\infty}^n $, then
\begin{equation} 
\label{eq:lem-T-bdd-S-bddsupp}
\Ex\sup_{s\in S, t \in T} \sum_{i\le m, j\le n} X_{i,j}s_it_j 
\lesssim_\alpha n^{1/p^*} \sup_{s\in S} \Bigl\|\sum_{i=1}^m X_{i,1}s_i \Bigr\|_{p^*} 
+ \bigl(m\wedge(l\Log m)\bigr)^{\gamman}l^{(1/q-1/2)\vee 0} b^{(2-p)/2}.
\end{equation}
\end{lem}

\begin{proof}
It suffices to prove \eqref{eq:lem-S-bdd-T-bddsupp}, since \eqref{eq:lem-T-bdd-S-bddsupp} follows by duality.

Without loss of generality we may assume that $k\le n$. Let $T_0$ be a $\frac 12$-net 
(with respect to $\ell_p^n$-metric) in $T$ of cardinality at most 
$5^n\wedge \bigl( {n\choose k} 5^k \bigr) \le 5^n \wedge (5n)^k=e^d$, where $d=n\wedge (k\ln(5n))$.
	
Then by \eqref{eq:maxk} we get
\begin{align}
\notag
\Ex\sup_{s\in S, t \in T} \sum_{i\le m, j\le n} X_{i,j}s_it_j 
& \leq 2\Ex \sup_{t\in {T_0}}\sup_{s\in S} \sum_{i\le m, j\le n} X_{i,j}s_it_j
 \leq 2e\sup_{t\in T_0} \Bigl(\Ex \sup_{s\in S}
\Bigl|\sum_{i\le m, j\le n} X_{i,j} s_it_j \Bigr|^{d} \Bigr)^{1/d}
\\
\label{eq:upp-lem-ST-bdd-1}
&\leq 2e\sup_{t\in T} \Bigl(\Ex \sup_{s\in S}\Bigl|\sum_{i\le m, j\le n} X_{i,j}s_it_j \Bigr|^{d} \Bigr)^{1/d}.
\end{align}
Fix $t\in T$. By \eqref{eq:comp-moments-regular} applied with $U=\{(s_it_j)_{i,j}\colon s\in S\}$ and $\rho =d$ we have
\begin{equation}
\label{eq:upp-lem-ST-bdd-2}
\Bigl(\Ex \sup_{s\in S}\Bigl|\sum_{i\le m, j\le n} X_{i,j}s_it_j \Bigr|^{d} \Bigr)^{1/d}
\lesssim_\alpha 
\Ex \sup_{s\in S}  \Bigl| \sum_{i\le m, j\le n} X_{i,j}s_it_j\Bigl|
+\sup_{s\in S}\Bigl\|\sum_{ i\le m,j\le n} X_{i,j}s_it_j\Bigr\|_d.
\end{equation}
Since $S\subset B_{q^*}^m$, 
\begin{equation}
\Ex \sup_{s\in S}  \Bigl|\sum_{i\le m, j\le n} X_{i,j}s_it_j \Bigr|
\leq \Bigl(\Ex\Bigl\|\Bigl(\sum_{j=1}^n X_{i,j}t_j\Bigr)_{i\leq m}\Bigr\|_q^q\Bigr)^{1/q}
=m^{1/q}\Bigl\|\sum_{j=1}^n X_{1,j}t_j\Bigr\|_q .
\label{eq:upp-lem-ST-bdd-3}
\end{equation}

Since $\alpha\ge \sqrt{2}$,  $\gamman = \frac 12\vee \log_2\alpha$, so
 by inequality \eqref{eq:compmomregalpha}
\begin{align}
\notag
\sup_{s\in S}\Bigl\|\sum_{ i\le m,j\le n} X_{i,j}s_it_j\Bigr\|_d 
& \lesssim_\alpha 
d^\gamman \sup_{s\in S,t\in T} \|s\|_2\|t\|_2
\leq d^\gamman\sup_{s\in S}\|s\|_\infty^{(2-q^*)/2} \|s\|_{q^*}^{q^*/2}
\sup_{t\in T}k^{(1/2-1/p)\vee 0}\| t\|_p 
 \\ 
\label{eq:upp-lem-ST-bdd-4}
&
\leq d^{\gamman}k^{(1/p^*-1/2)\vee 0}a^{(2-q^*)/2}.
\end{align}
Inequalities \eqref{eq:upp-lem-ST-bdd-1}-\eqref{eq:upp-lem-ST-bdd-4} yield \eqref{eq:lem-S-bdd-T-bddsupp}.
\end{proof}

 In the sequel $(g_{i,j})_{i\le m, j\le n}$ are iid standard Gaussian random 
variables.

\begin{lem}
\label{lem:comwithgauss}
Let $(X_{i,j})_{i\leq m,j\leq n}$ be iid symmetric random variables satisfying \eqref{alphareg}
and $\Ex X_{i,j}^2=1$.  Let $\gamman=\log_{2} \alpha$. 
Then for any nonempty bounded sets $S\subset \er^m$ and $T\subset  \er^n$ we have
\begin{align*}
\Ex\sup_{s\in S, t \in T} \sum_{i\le m, j\le n} X_{i,j}s_it_j
\lesssim \Log^\gamman(mn)\Ex\sup_{s\in S, t \in T} \sum_{i\le m, j\le n} g_{i,j}s_it_j.
\end{align*}
\end{lem}

\begin{proof}
Since $X_{i,j}$'s are independent and symmetric,
 $(X_{i,j})_{i\le m, j\le n}$ has the same distribution as  $(\varepsilon_{i,j}|X_{i,j}|)_{i\le m, j\le n}$, where $(\varepsilon_{i,j})_{i\le m, j\le n}$ are iid symmetric $\pm 1$ random variables independent of $X_{i,j}$'s. By the contraction principle 
\begin{align}
\notag
\Ex\sup_{s\in S, t \in T} \sum_{i\le m, j\le n} X_{i,j}s_it_j 
& = \Ex\sup_{s\in S, t \in T} \sum_{i\le m, j\le n} \varepsilon_{i,j} |X_{i,j}|s_it_j 
\\ 
\label{lem-comwithgauss-1}
& \le \Ex\max_{i\le m, j\le n} |X_{i,j}| \cdot \Ex \sup_{s\in S, t \in T} \sum_{i\le m, j\le n} \varepsilon_{i,j}s_it_j .
\end{align} 
Moreover, by \eqref{eq:maxk} and regularity assumption \eqref{alphareg} we have
\begin{equation}
\label{lem-comwithgauss-2}
\Ex \max_{i\le m, j\le n} |X_{i,j}| 
\leq e \|X_{1,1}\|_{\Log(mn)} 
 \lesssim   \Log^\gamman(mn) \|X_{1,1}\|_2 
= \Log^\gamman(mn).
\end{equation} 
Jensen's inequality yields
\begin{align} 
\Ex \sup_{s\in S, t \in T} \sum_{i\le m, j\le n} \varepsilon_{i,j}s_it_j
\sim \Ex \sup_{s\in S, t \in T} \sum_{i\le m, j\le n} \varepsilon_{i,j}\Ex |g_{i,j}|s_it_j
 \lesssim \Ex \sup_{s\in S, t \in T} \sum_{i\le m, j\le n} g_{i,j} s_it_j.
\label{lem-comwithgauss-3}
\end{align}
Inequalities \eqref{lem-comwithgauss-1}-\eqref{lem-comwithgauss-3} yield the assertion.
\end{proof}

The next result is an immediate consequence of the contraction principle (see also \eqref{lem-comwithgauss-1} together with \eqref{lem-comwithgauss-3}), but turns out to be helpful.

\begin{lem}
\label{lem:comwithgauss2}
Let $(X_{i,j})_{i\leq m,j\leq n}$ be centered random variables. Then 
\[
\Ex\sup_{s\in S, t \in T} \sum_{i\le m, j\le n} X_{i,j}s_it_j
\lesssim \max_{i,j}\|X_{i,j}\|_\infty \Ex\sup_{s\in S, t \in T} \sum_{i\le m, j\le n} g_{i,j}s_it_j.
\]
\end{lem}

Let us recall Chevet's inequality from \cite{Ch}:
\begin{equation}
\label{eq:Chevet}
\Ex\sup_{s\in S, t \in T} \sum_{i\le m, j\le n} g_{i,j}s_it_j
\lesssim 
\sup_{s\in S}\|s\|_2\Ex\sup_{t\in T}\sum_{j\leq n}g_jt_j
+\sup_{t\in T}\|t\|_2\Ex\sup_{s\in S}\sum_{i\leq m}g_is_i.
\end{equation}
We use it to derive the following two lemmas.

\begin{lem} 	
\label{lem:viaChevet}
Let $q\geq 2$, $p\geq 1$, $S\subset \{s\in B_{q^*}^m\colon\ |\supp(s)|\le l\}\cap aB_{\infty}^m$, and $T\subset B_{p}^n$. Then
\begin{align*}
\Ex\sup_{s\in S, t \in T} \sum_{i\le m, j\le n} g_{i,j}s_it_j
\lesssim \sqrt{p^*}a^{(2-q^*)/2}n^{1/p^*}+n^{(1/p^*-1/2)\vee 0}\sqrt{\Log m}\,l^{1/q}.
\end{align*}
If we assume additionally that  $l=m$, $p^*\ge 2$, and $T\subset bB_\infty^n$, then
\begin{align}\label{eq:viaChevet-ST-bdd}
\Ex\sup_{s\in S, t \in T} \sum_{i\le m, j\le n} g_{i,j}s_it_j
\lesssim \sqrt{p^*}a^{(2-q^*)/2}n^{1/p^*}+\sqrt{q}b^{(2-p)/2}m^{1/q}.
\end{align}
\end{lem}

\begin{proof}
We have
\[
\sup_{t\in T}\|t\|_2\leq \sup_{t\in B_p^n}\|t\|_2=n^{(1/p^*-1/2)\vee 0},
\]
\[
\sup_{s\in S}\|s\|_2\leq \sup_{s\in S}\|s\|_{q^*}^{q^*/2}\|s\|_\infty^{(2-q^*)/2}\leq a^{(2-q^*)/2},
\] 
\begin{equation*}
\Ex\sup_{t\in T}\sum_{j=1}^n g_jt_j\leq 
 \Ex\sup_{t\in B_p^n}\sum_{j=1}^n g_jt_j=
\Ex\|(g_j)_{j=1}^n\|_{p^*}
\leq (\Ex\|(g_j)_{j=1}^n\|_{p^*}^{p^*})^{1/p^*}=\|g_1\|_{p^*}n^{1/p^*}\leq \sqrt{p^*} n^{1/p^*},
\end{equation*}
and
\[
\Ex\sup_{s\in S}\sum_{i=1}^m g_is_i\leq \Ex\sup_{I\subset [m],|I|\leq l}\Bigl(\sum_{i\in I}|g_i|^q\Bigr)^{1/q}
\leq l^{1/q}\Ex\max_{i\le m}|g_i|\lesssim l^{1/q}\sqrt{\Log m}.
\]
The first assertion follows by Chevet's inequality \eqref{eq:Chevet} and the four bounds above.

In the case when  $l=m$, $p^*\ge 2$, and $T\subset bB_\infty^n$ we  use a different bound for $\sup_{t\in T}\|t\|_2$, namely
\[
\sup_{t\in T}\|t\|_2\leq \sup_{t\in T}\|t\|_{p}^{p/2}\|t\|_\infty^{(2-p)/2}\leq b^{(2-p)/2},
\] 
and for $\Ex\sup_{s\in S}\sum_{i=1}^m g_is_i$, namely
\[
\Ex\sup_{s\in S}\sum_{i=1}^m g_is_i  \le \Ex\sup_{s\in B_{q^*}^m}\sum_{i=1}^m g_is_i \le \sqrt q m^{1/q}.\qedhere
\]
\end{proof}

The next lemma is a slight modification of the previous one.

\begin{lem} 	
\label{lem:viaChevet2}
Let  $2\leq p^*,q \leq\gamma$, $S\subset \{s\in B_{q^*}^m\colon\ |\supp(s)|\le l\}\cap aB_{\infty}^m$ and $T\subset B_{p}^n$. Then
\begin{align*}
\Ex\sup_{s\in S, t \in T} \sum_{i\le m, j\le n} g_{i,j}s_it_j
\lesssim \sqrt{\gamma} \Bigl(a^{(2-q^*)/2}n^{1/p^*}+\sqrt{\Log (m/l)}\,l^{1/q}\Bigr).
\end{align*}
\end{lem}

\begin{proof}
We proceed as in the previous proof, observing that $\sqrt{p^*}\leq  \sqrt{\gamma}$ and, 
 by \cite[Lemmas 19 and 23]{LSChevet},
\[
\Ex\sup_{I\subset [m],|I|\leq l}\Bigl(\sum_{i\in I}|g_i|^q\Bigr)^{1/q}
 \lesssim \sqrt{\gamma\vee \Log (m/l)}\,l^{1/q}.\qedhere
\]
\end{proof}

The next proposition is a consequence of the $\ell_2^n\to \ell_2^m$ bound from \cite{La}.

\begin{lem}
\label{lem:vial2tol2}
Let $(X_{i,j})_{i\leq m,j\leq n}$, be  be iid symmetric random variables satisfying \eqref{alphareg}  with $\alpha \ge  \sqrt{2}$ and $\Ex X_{i,j}^2=1$. Then for $M>0$,
\[
\Ex\bigl\|\bigl(X_{i,j}I_{\{|X_{i,j}|\geq M}\}\bigr)_{i\leq m,j\leq n}\bigr\|_{\ell_2^n\to\ell_2^m}
\lesssim_\alpha (\sqrt{n}+\sqrt{m})\exp\Bigl(-\frac{\ln \alpha}{10} M^{1/\log_2\alpha}\Bigr).
\]
\end{lem}

\begin{proof}
By \cite[Theorem~2]{La} we have
\begin{align*}
\Ex\bigl\|\bigl(X_{i,j}I_{\{|X_{i,j}|\geq M\}}\bigr)_{i\leq m,j\leq n}\bigr\|_{\ell_2^n\to\ell_2^m}
&\leq \max_{i\leq m}\Bigl(\sum_{j\leq n}\Ex X_{i,j}^2I_{\{|X_{i,j}|\geq M\}}\Bigr)^{1/2}
+\max_{j\leq n}\Bigl(\sum_{i\leq m}\Ex X_{i,j}^2I_{\{|X_{i,j}|\geq M\}}\Bigr)^{1/2}
\\
&\qquad+\Bigr(\sum_{i\leq m,j\leq n}\Ex X_{i,j}^4I_{\{|X_{i,j}|\geq M\}}\Bigr)^{1/4}.
\end{align*}

Regularity condition \eqref{alphareg} and the normalization $\|X_{i,j}\|_2=1$ yields 
$\|X_{i,j}\|_{\rho}\leq \alpha^{\log_2  \rho}$ for all $\rho\geq 1$.
Thus, for all $\rho\geq 4$,
\[
\bigl(\Ex X_{i,j}^2I_{\{|X_{i,j}|\geq M\}}\bigr)^{1/2}\leq 
\bigl(\Ex X_{i,j}^4I_{\{|X_{i,j}|\geq M\}}\bigr)^{1/4}\leq (M^{4-\rho}\Ex X_{i,j}^{\rho})^{1/4}
\leq M\Bigl(\frac{\alpha^{\log_2 \rho}}{M}\Bigr)^{\rho/4} .
\]
Let us choose $\rho:=\frac{1}{2}M^{1/\log_2\alpha}$. If $M\geq \alpha^3$, then $\rho\geq 4$, so
\[
M\Bigl(\frac{\alpha^{\log_2 \rho}}{M}\Bigr)^{\rho/4} = M\alpha^{-\rho/4}
=M\exp\Bigl(-\frac{\ln \alpha}{8} M^{1/\log_2\alpha}\Bigr)
\lesssim_\alpha \exp\Bigl(-\frac{\ln\alpha}{10} M^{1/\log_2\alpha}\Bigr).
\]
If $M\leq \alpha^3$, then
\[
\bigl(\Ex X_{i,j}^2I_{\{|X_{i,j}|\geq M\}}\bigr)^{1/2}\leq 
 \bigl(\Ex X_{i,j}^4I_{\{|X_{i,j}|\geq M\}}\bigr)^{1/4}
\leq
 (\Ex X_{ij}^4)^{1/4}\leq \alpha
\lesssim_\alpha \exp\Bigl(-\frac{\ln \alpha}{10} M^{1/\log_2\alpha}\Bigr). \qedhere
\]
\end{proof}


\subsection{Case  $p^*\gtrsim_\alpha \Log m$ or $q\gtrsim_\alpha \Log n$}	
\label{sect:above-nonmatching-logs}

\begin{prop}
Theorem~\ref{thm:rect} holds in the case $p^*\gtrsim_\alpha \Log m$ or $q\gtrsim_\alpha \Log n$.
\end{prop}

\begin{proof}
 Without loss of generality we may assume that $\|X_{i,j}\|_2=1$. 
 By Remark~\ref{rem:symmetrization} it suffices to assume that  $X_{i,j}$'s are symmetric and $\alpha\ge \sqrt{2}$, 
 and by duality  \eqref{eq:duality} it suffices to consider the case $q\ge C_0(\alpha) \Log n$, where
\[
C_0(\alpha)=8\gamman  = 8\log_2 \alpha.
\]  
In particular $q\ge 4$, so  $q^*\le 4/3$.
By Subsection~\ref{seq:case-leqLog} it suffices to consider the case $p^*\le \Log n$.

Define 
\[
S_1=B_{q^*}^m\cap e^{-q}B_\infty^m,
\quad S_2=\{s\in B_{q*}^m\colon |\mathrm{supp}(s)|\leq e^{qq^*}\}.
\]
Then $B_{q^*}^m \subset S_1+S_2$.

If $s\in S_2$, then 
\[
\|s\|_1\le \|s\|_{q^*} |\supp(s)|^{-1/q^*+1}\le e^{q^*}\le e^{4/3},
\]
so $S_2\subset e^{4/3} B_{\infty^*}^m$. Thus, Proposition~\ref{prop:above-log} and  \eqref{eq:compmomregalpha} imply
\begin{align*}
\Ex\sup_{s\in S_2, t \in B_{p}^n}\sum_{i\le m, j\le n} X_{i,j}s_it_j
& \lesssim 
\Ex\bigl\|(X_{i,j})_{i\leq m,j\leq n}\bigr\|_{\ell_p^n\rightarrow\ell_\infty^m}
\\ 
&\sim_\alpha
\sup_{t\in B_p^n}\Bigl\|\sum_{j\leq n}t_jX_{1,j}\Bigr\|_{\Log m}+n^{1/p^*}\|X_{1,1}\|_{ p^*}
\\ 
& \lesssim_{\alpha} 
\Bigl( 1\vee \frac {\Log m} q\Bigr)^{\beta} \sup_{t\in B_p^n}\Bigl\|\sum_{j\leq n}t_jX_{1,j}\Bigr\|_{q}
+n^{1/p^*}\|X_{1,1}\|_{ p^*}.
\end{align*}
Since the function $0< q\mapsto \frac 1{ q} \ln m +\beta \ln  q$ attains
its minimum at $ q=\ln m/\beta$, where the function's value is equal to 
$-\beta \ln(\beta/e)+\beta\ln\ln m $, we have $(\Log m/q)^\beta \lesssim_\alpha m^{1/q}$.   
Hence, the previous upper bound yields
\begin{equation}
\label{eq:cor-gtr-nonmatching-log-1}
\Ex\sup_{s\in S_2, t \in B_{p}^n}\sum_{i\le m, j\le n} X_{i,j}s_it_j 
\lesssim_\alpha m^{1/q}\sup_{t\in B_p^n}\Bigl\|\sum_{j\leq n}t_jX_{1,j}\Bigr\|_{q}
+n^{1/p^*} \sup_{s\in B_{q^*}^m}\Bigl\|\sum_{i=1}^{m} s_iX_{i,1}\Bigr\|_{ p^*}.
\end{equation}

Moreover, \eqref{eq:lem-S-bdd-T-bddsupp} from Lemma~\ref{lem:S-bdd-T-bddsupp} applied with $S=S_1$, $T=B_p^n$, $a=e^{-q}$,  and $k=n$, together with the inequality $q^*\le 4/3$, implies that
\begin{equation}
\label{eq:cor-gtr-nonmatching-log-2}
\Ex\sup_{s\in S_1, t \in B_{p}^n}\sum_{i\le m, j\le n} X_{i,j}s_it_j 
\lesssim_\alpha 
m^{1/q}\sup_{t\in B_p^n}\Bigl\|\sum_{j\leq n}t_jX_{1,j}\Bigr\|_{q}
+ n^{\beta+ ((1/p^*-1/2)\vee 0)} e^{-q/3}.
\end{equation}
Since $q\geq C_0(\alpha)\Log n\geq  3(\beta+1/2)\ln n$ 
 and $\|X_{1,1}\|_{p^*}\gtrsim_\alpha \|X_{1,1}\|_2=1$, 
inequalities \eqref{eq:cor-gtr-nonmatching-log-1} and \eqref{eq:cor-gtr-nonmatching-log-2} 
yield the assertion.
\end{proof}


\subsection{Case  $p^*,q\ge 3$}	
\label{sect:above-4}

 By Subsection~\ref{seq:case-leqLog} we may assume that $p^*\leq \Log n$ and $q\leq \Log m$. 
In this subsection we restrict ourselves to to the case  $p^*,q\ge 3$. However, 
similar proofs work also in the range $p^*, q\ge 2+\ve$, where $\ve >0$ is arbitrary --- in this 
case the constants in upper bounds depend also on $\ve$ and blow up when $\ve$ approaches $0$. 
 If $p^*$ or $q$ lies above and close to $2$, then we need 
different arguments, which we show in next subsections.

\begin{lem}	
\label{lem:ST-bdd}
Assume that $3\leq p^*, q \leq \Log(mn) $, $(X_{i,j})_{i\leq m,j\leq n}$ are iid symmetric
random variables satisfying \eqref{alphareg} with $\alpha \ge  \sqrt{2}$, 
$\Ex X_{i,j}^2=1$, $S\subset B_{q^*}^m \cap \Log^{-8\gamman}(mn) B_\infty^m$, and 
$T\subset B_p^n \cap \Log^{-8\gamman}(mn) B_\infty^n$, where $\gamman= \log_2 \alpha$. Then 
\[
\Ex\sup_{s\in S, t \in T} \sum_{i\le m, j\le n} X_{i,j}s_it_j 
\lesssim  m^{1/q}+ n^{1/p^*}.
\]
\end{lem}

\begin{proof}

Lemma \ref{lem:comwithgauss} and inequality \eqref{eq:viaChevet-ST-bdd} yield
\begin{align}
\Ex\sup_{s\in S, t \in T} \sum_{i\le m, j\le n} X_{i,j}s_it_j
\lesssim 
\label{lem-both-bdd-upper-3}
\Log^{1/2+\gamman}(mn) \bigl(m^{1/q} \Log^{-4\gamman(2-p)}(mn)+ n^{1/p^*} \Log^{-4\gamman(2-q^*)}(mn)\bigr).
\end{align}

Since $p^*\ge 3$, $(2-p)\ge 1/2$, so
\[
 \Log^{-4\gamman(2-p)}(mn) 
\le \Log^{-2\gamman}(mn) \leq \Log^{-\gamman-1/2}(mn),
\] 
and similarly 
\[
 \Log^{-4\gamman(2-q^*)}(mn) 
\le \Log^{-\gamman-1/2}(mn),
\] 
This together with bound  \eqref{lem-both-bdd-upper-3} implies the assertion.
\end{proof}

Now we are ready to prove  the upper bound in Theorem~\ref{thm:rect} in the case when $p^*,q$ 
are separated from $2$.

\begin{prop}	\label{prop:q-above-4}
Let $(X_{i,j})_{i\leq m,j\leq n}$ be iid symmetric random variables such that \eqref{alphareg} holds with $\alpha\ge  \sqrt{2}$.
Then  the upper bound in Theorem~\ref{thm:rect}  holds whenever 
$3\leq q\leq \Log m$ and $3\leq p^* \leq \Log n$.
\end{prop}

\begin{proof}
Without loss of generality we assume that $\Ex X_{i,j}^2=1$ and that $q\ge p^*$ (the opposite case follows by duality   \eqref{eq:duality}).

 Recall that $\gamman=\log_2\alpha\geq 1/2$
and let us consider the following subsets of balls  $B_{q^*}^m$ and $B_p^n$:
\[
S_1=B_{q^*}^m\cap e^{-q}B_\infty^m,
\quad S_2=\{s\in B_{q*}^m\colon |\mathrm{supp}(s)|\leq e^{qq^*}\},
\]
\[
S_3=B_{q^*}^m\cap \Log^{- 8\gamman}(mn) B_\infty^m,
\quad S_4=\{s\in B_{q*}^m\colon |\mathrm{supp}(s)|\leq \Log^{ 8\gamman q^*}(mn)\},
\]
\[
T_1=B_p^n\cap e^{-p^*}B_\infty^n,\quad 
T_2=\{t\in B_p^n\colon |\mathrm{supp}(t)|\leq e^{pp^*}\},
\]
and
\[
T_3=B_p^n\cap  \Log^{- 8\gamman}(mn)B_\infty^n,\quad 
T_4=\{t\in B_p^n\colon |\mathrm{supp}(t)| \leq \Log^{  8\gamman p}(mn)\}.
\]
Note that $B_{q^*}^m\subset S_1+S_2$, $B_{q^*}^m\subset S_3+S_4$, $B_{p}^n\subset T_1+T_2$, and $B_{p}^n\subset T_3+T_4$. In particular
\begin{align}
\label{eq:prop-above4-0} 
\bigl\|&(X_{i,j})_{i\leq m,j\leq n}\bigr\|_{\ell_p^n\rightarrow\ell_q^m}
 = \sup_{s\in B_{q^*}^m, t \in B_{p}^n}\sum_{i\le m, j\le n} X_{i,j}s_it_j 
\\
\notag
&\le \sup_{s\in S_1, t \in T_1}\sum_{i\le m, j\le n} X_{i,j}s_it_j 
 + \sup_{s\in S_2, t \in B_{p}^n}\sum_{i\le m, j\le n} X_{i,j}s_it_j +
\sup_{s\in B_{q^*}^m, t \in T_2}\sum_{i\le m, j\le n} X_{i,j}s_it_j.
\end{align}

If $s\in S_2$, then 
\[
\|s\|_1\le \|s\|_{q^*} |\supp(s)|^{-1/q^*+1}\le e^{q^*}\le e^{3/2}<5,
\]
so $S_2\subset 5B_1^m =5B_{\infty^*}^m$ 
and we may proceed as in the proof of \eqref{eq:cor-gtr-nonmatching-log-1} to get
\begin{equation}
\label{eq:prop-above4-1}
\Ex\sup_{s\in S_2, t \in B_{p}^n}\sum_{i\le m, j\le n} X_{i,j}s_it_j 
\lesssim_\alpha 
m^{1/q}\sup_{t\in B_p^n}\Bigl\|\sum_{j\leq n}t_jX_{1,j}\Bigr\|_{q}
+n^{1/p^*} \sup_{s\in B_{q^*}^m}\Bigl\|\sum_{i=1}^{m} s_iX_{i,1}\Bigr\|_{ p^*}
\end{equation}
and, by duality,
\begin{equation}
\label{eq:prop-above4-2}
\Ex\sup_{s\in B_{q^*}^m, t \in T_2}\sum_{i\le m, j\le n} X_{i,j}s_it_j 
\lesssim_\alpha 
m^{1/q}\sup_{t\in B_p^n}\Bigl\|\sum_{j\leq n}t_jX_{1,j}\Bigr\|_{q}
+n^{1/p^*} \sup_{s\in B_{q^*}^m}\Bigl\|\sum_{i=1}^{m} s_iX_{i,1}\Bigr\|_{ p^*}.
\end{equation}

Bounds \eqref{eq:prop-above4-0}-\eqref{eq:prop-above4-2} imply that it suffices to prove that
\begin{equation}
\label{eq:prop-above4-3}
\sup_{s\in S_1, t \in T_1}\sum_{i\le m, j\le n} X_{i,j}s_it_j
\lesssim_\alpha
m^{1/q}\sup_{t\in B_p^n}\Bigl\|\sum_{j=1}^nt_jX_{1,j}\Bigr\|_{ q}
+n^{1/p^*}\sup_{s\in B_{q^*}^m}\Bigl\|\sum_{i=1}^{m} s_iX_{i,1}\Bigr\|_{ p^*}.
\end{equation}
Recall that $q\ge p^*\ge 3$. Let us consider three cases.

\textbf{Case 1}, when $q,p^*\ge  60\gamman^2\Log\Log(mn)$. 
Then $e^{-q},e^{-p^*} \le \Log^{- 8\gamman}(mn)$, 
so $S_1\subset S_3$ and $T_1\subset T_3$. 
Thus, \eqref{eq:prop-above4-3} follows by Lemma~\ref{lem:ST-bdd}.
	
\textbf{Case 2}, when $q\ge   60\gamman^2\Log\Log(mn)\ge p^*$. Then $S_1\subset S_3$ and 
$T_1\subset B_p^n \subset T_3+ T_4$, so
\[
\Ex \sup_{s\in S_1, t \in T_1}\sum_{i\le m, j\le n} X_{i,j}s_it_j 
\le  \Ex \sup_{s\in S_3, t \in T_3}\sum_{i\le m, j\le n} X_{i,j}s_it_j 
+ \Ex  \sup_{s\in S_1, t \in  T_4}\sum_{i\le m, j\le n} X_{i,j}s_it_j .
\] 
The first term on the right-hand side may be  bounded properly by Lemma~\ref{lem:ST-bdd}. 
In order to estimate the second term we apply \eqref{eq:lem-S-bdd-T-bddsupp} from 
Lemma~\ref{lem:S-bdd-T-bddsupp}  with $a=e^{-q}$
and $k=\lfloor \Log^{ 12\gamman }(mn)  \rfloor\ge \lfloor \Log^{8\gamman p}(mn)  \rfloor$ (the 
inequality follows by $p\le 3^*=\frac 32$) to get
\[
\Ex \sup_{s\in S_1, t \in  T_4}\sum_{i\le m, j\le n} X_{i,j}s_it_j
\lesssim_\alpha m^{1/q} \sup_{t\in  T_4} \Bigl\|\sum_{j=1}^n X_{1,j}t_j \Bigr\|_q + 
(\Log^{14\gamman} (mn))^\gamman e^{-q(2-q^*)/2}.
\]
Since $q^*\le 3^*=3/2$, we have 
\[
(\Log^{14\gamman} (mn))^\gamman e^{-q(2-q^*)/2} \le  \Log^{14\gamman^2} (mn) e^{-q/4} 
\le 1,
\]
so \eqref{eq:prop-above4-3} holds.
	
\textbf{Case 3}, when $  60\gamman^2 \Log\Log(mn) \ge q, p^*$. 
Since $T_1\subset T_3+T_4$ and $S_1\subset S_3+S_4$, we have
\begin{align*}
\Ex \sup_{s\in S_1, t \in T_1}\sum_{i\le m, j\le n} X_{i,j}s_it_j 
&\le  
\Ex \sup_{s\in S_3, t \in T_3}\sum_{i\le m, j\le n} X_{i,j}s_it_j 
+\Ex \sup_{s\in B_{q^*}^m, t \in T_4}\sum_{i\le m, j\le n} X_{i,j}s_it_j 
\\ 
&\qquad+ \Ex  \sup_{s\in S_4, t \in  B_{p}^n }\sum_{i\le m, j\le n} X_{i,j}s_it_j .
\end{align*}
The first term on the right-hand side may be  bounded by Lemma~\ref{lem:ST-bdd}. 
Now we estimate the second term --- the third one may be bounded similarly (by using \eqref{eq:lem-T-bdd-S-bddsupp} from Lemma~\ref{lem:S-bdd-T-bddsupp} instead of \eqref{eq:lem-S-bdd-T-bddsupp}). 
By \eqref{eq:lem-S-bdd-T-bddsupp} applied with $a=1$ and 
$k=\lfloor \Log^{12\gamman }(mn)  \rfloor\ge \lfloor \Log^{8\gamman p}(mn)  \rfloor$ 
we have
\begin{equation*}
\Ex \sup_{s\in B_{q^*}^m, t \in T_4}\sum_{i\le m, j\le n} X_{i,j}s_it_j 
\lesssim_\alpha 
m^{1/q} \sup_{t\in  T_4} \Bigl\|\sum_{j=1}^n X_{1,j}t_j \Bigr\|_q 
+ \Log^{14\gamman^2} (mn).
\end{equation*}

For a fixed $\gamman= \log_2 \alpha\ge  1/2$ there exists $C(\gamman)\ge 3$ such that 
for every $x\ge C(\gamman) \eqqcolon C_0(\alpha)$ we have 
$28\gamman^2 \ln x\le x/(60\gamman^2\ln x)$.
Hence, if $mn\ge e^{C_0(\alpha)}$ and $p^*\le q\le 60\gamman^2 \Log\Log(mn)$, then 
\[
14\gamman^2 \ln \Log (mn) \le \frac 12 \ln (mn)/q \le \frac 12 (\ln m /q +\ln n/p^*) 
\le \max\{\ln m /q ,\ln n/p^* \},
\]
so for every $m,n\in \en$,
\[
\Log^{ 14\gamman^2}(mn)\lesssim_\alpha  \max\{m^{1/q},n^{1/p^*} \},
\]
and \eqref{eq:prop-above4-3} follows.
\end{proof}

\subsection{Case $q\geq  24\gamman\geq 3\geq p^*$} \label{sect:qge8betage3gep*}

In this subsection we assume (without loss of generality -- see Remark~\ref{rem:symmetrization}) 
that $X_{i,j}$ are iid symmetric random variables  satisfying \eqref{alphareg} with 
$\alpha\ge  \sqrt{2}$. We also use the notation $\gamman=\log_2 \alpha\ge  1/2$. 
In particular $24\gamman \geq 3$  and $q^*\leq 3/2$ whenever $q\ge 24\gamman$.
Once we prove the upper bound in the case $q\ge  24\gamman \geq p^*$, the upper bound in the case $p^*\geq  24\gamman \geq q$ follows by duality   \eqref{eq:duality}. 
By Subsections~\ref{seq:case-leqLog} and \ref{sect:above-nonmatching-logs} it suffices to consider the case $\Log m \wedge C(\alpha) \Log n\ge q$.
In this case Theorem~\ref{thm:rect} follows by the following two lemmas.

\begin{lem}
\label{lem:obs1}
If $\Log m \geq q\geq 3\geq p^*$,
$n^{1/3}\geq m^{1/q}q^\beta$, and $\|X_{1,1}\|_2=1$, then 
\[
\Ex\|(X_{i,j})_{i\le m, j\le n}\|_{\ell_p^n\to\ell_q^{m}}\lesssim n^{1/p^*}.
\]
\end{lem}

\begin{proof}
By \eqref{eq:compmomregalpha} we get
\[
\sup_{t\in B_{3/2}^n}\Bigr\|\sum_{i=1}^mt_iX_{i,1}\Bigr\|_{q}
\leq q^\beta \sup_{t\in B_{3/2}^n}\Bigr\|\sum_{i=1}^mt_iX_{i,1}\Bigr\|_{2}
= q^\beta \sup_{t\in B_{3/2}^n}\|t\|_2=q^\beta.
\]
This together with the assumption $n^{1/3}\geq m^{1/q}q^\beta$ and the estimate in the case 
$p^*=3 \le q$ (already obtained in Subsection~\ref{sect:above-4}) gives
$\Ex\|(X_{i,j})\|_{\ell_{3/2}^n\to\ell_q^{m}}\lesssim n^{1/3}$.
Therefore, for  every $p^*\leq 3$,
\[
\Ex\|(X_{i,j})\|_{\ell_{p}^n\to\ell_q^{m}}
\leq \|\operatorname{Id}\|_{\ell_{p}^n\to\ell_{3/2}^n}\Ex\|(X_{i,j})\|_{\ell_{3/2}^n\to\ell_q^{m}}
\lesssim n^{2/3-1/p}n^{1/3}=n^{1/p^*}. \qedhere
\]
\end{proof}

\begin{lem}\label{lem:q-ge-8beta-ge-p*}
Assume that $\Log m \wedge C(\alpha) \Log n \geq q\geq  24\gamman\geq 3\geq p^*$ and 
$q^\gamman m^{1/q}\geq n^{1/3}$. Then the upper bound in Theorem~\ref{thm:rect}  holds.
\end{lem}

\begin{proof}
Without loss of generality we may assume that $\Ex X_{i,j}^2=1$ and $C(\alpha)\ge 2$.
Let 
\[
\widetilde{S}_1
=\bigl\{s\in B_{q^*}^m\colon\ |\supp(s)|\leq \Log^{ 4\gamman q^*}(mn)\bigr\},
\quad
S_1=B_{q^*}^m\cap \Log^{-4\gamman}(mn)B_{\infty}^m.
\]
Then $B_{q^*}^m\subset S_1+\widetilde{S}_1$.

If $\Log m\le C^2(\alpha) \Log^2n$, then inequality \eqref{eq:lem-T-bdd-S-bddsupp} from 
Lemma~\ref{lem:S-bdd-T-bddsupp}  (applied with  $b=1$, $p\wedge 2$ instead of $p$ and 
$l= \Log(mn)^{4\gamman q^*}\leq \Log(mn)^{6\gamman}$) yields
\begin{align*}
\Ex\sup_{s\in \widetilde{S}_1, t \in B_{ p\wedge 2}^n} \sum_{i\le m, j\le n} X_{i,j}s_it_j
& \lesssim_{\alpha} n^{1/(p^*\vee 2)}
\sup_{s\in \widetilde{S}_1} \Bigl\|\sum_{i=1}^m X_{i,1}s_i \Bigr\|_{  p^*\vee 2}
+ (\Log n)^{C_1(\alpha)}
\\ & \lesssim_{\alpha}  
n^{1/(p^*\vee 2)} \sup_{s\in B_{q^*}^m} \Bigl\|\sum_{i=1}^m X_{i,1}s_i \Bigr\|_{p^*}
+ n^{1/3}
\\
& \lesssim_\alpha 
n^{1/ (p^*\vee 2)} \sup_{s\in B_{q^*}^m} \Bigl\|\sum_{i=1}^m X_{i,1}s_i \Bigr\|_{p^*} .
\end{align*}
In the case $\Log m \ge C^2(\alpha) \Log^2n$ we have $m^{1/q}\ge e^{\Log m/(C(\alpha)\Log n)}\ge e^{(\Log m)^{1/2}}$, so now inequality \eqref{eq:lem-T-bdd-S-bddsupp} yields
\begin{align*}
\Ex\sup_{s\in \widetilde{S}_1, t \in B_{ p\wedge 2}^n} \sum_{i\le m, j\le n} X_{i,j}s_it_j
& \lesssim_{\alpha} n^{1/(p^*\vee 2)}
\sup_{s\in \widetilde{S}_1} \Bigl\|\sum_{i=1}^m X_{i,1}s_i \Bigr\|_{ p^*\vee 2} 
+ (\Log m)^{C_2(\alpha)}
\\ & \lesssim_{\alpha}    n^{1/ (p^*\vee 2)} 
\sup_{s\in B_{q^*}^m} \Bigl\|\sum_{i=1}^m X_{i,1}s_i \Bigr\|_{ p^*} +m^{1/q}.
\end{align*}
Thus, in any case
\begin{align}
\label{eq:lem-q-ge-8beta-ge-p*-1}
\Ex\sup_{s\in \widetilde{S}_1, t \in B_{p}^n} \sum_{i\le m, j\le n} X_{i,j}s_it_j
&\leq
n^{(1/p^*-1/2)\vee 0}\, \Ex\sup_{s\in \widetilde{S}_1, t \in B_{ p\wedge 2}^n} \sum_{i\le m, j\le n} X_{i,j}s_it_j
\notag
\\
&\lesssim_\alpha n^{1/p^*}
\sup_{s\in B_{q^*}^m} \Bigl\|\sum_{i=1}^m X_{i,1}s_i \Bigr\|_{ p^*} 
+m^{1/q}n^{(1/p^*-1/2)\vee 0}.
\end{align}

Let 
\begin{align*}
S_2&=\bigl\{s\in B_{q^*}^m\colon\ |\supp(s)|\leq m\Log^{-q(\gamman+1)} (mn) \bigr\}\cap \Log^{-4\gamman}(mn)B_{\infty}^m,
\\
S_3&=B_{q^*}^m\cap m^{-1/q^*}\Log^{(\gamman+1)q/q^*}(mn)B_\infty^m.
\end{align*}
Then $S_1\subset S_2+S_3$.

Lemmas \ref{lem:comwithgauss} and \ref{lem:viaChevet}
(applied with $l=m\Log^{-q(\beta+1)}(mn)$ and $a=\Log^{-4\beta}(mn)$), and inequality 
$q^* \leq  \frac 32$  yield
\begin{align}\label{eq:lem-q-ge-8beta-ge-p*-2}
\Ex\sup_{s\in S_2, t \in B_p^n} &\sum_{i\le m, j\le n} X_{i,j}s_it_j
\notag
\\
\notag
&\leq \Log^{\gamman}(mn)\Big(\Log^{-2\gamman(2-q^*)}(mn)n^{1/p^*}+n^{(1/p^*-1/2)\vee 0}m^{1/q}\Log^{-\gamman-1/2}(mn)\Bigr)
\\
&\leq n^{1/p^*}+n^{(1/p^*-1/2)\vee 0}m^{1/q}.
\end{align}

Moreover, if $\Log m\le C^2(\alpha) \Log^2n$, then inequalities 
$n^{1/3}\leq m^{1/q}q^\gamman\leq m^{1/q}\Log^\gamman m$ and 
$q/(3q^*) \ge  4\gamman+q/(12q^*)$ imply
\[
m^{1/q^*}\Log^{-(\gamman+1)q/q^*}(mn)
\geq n^{q/(3q^*)}\Log^{-\gamman q/q^*} m\Log^{-(\gamman+1)q/q^*}(mn)
\gtrsim_\alpha n^{ 4\gamman},
\]
and if $\Log m\ge C^2(\alpha) \Log^2n\ge\Log^2n$, then
\begin{align*}
m^{1/q^*}\Log^{-(\gamman+1)q/q^*}(mn)
&\geq e^{\Log m/q^* } \Log^{-C_3(\alpha)q}m
\ge \exp\bigl( (\Log m)/2 - C_4(\alpha) \Log n \cdot \ln(\Log m)\bigr)
\\ 
& \gtrsim_{\alpha} e^{(\Log^2n)/4}\gtrsim_{\alpha} n^{ 4 \gamman}.
\end{align*}
Since $q^* \leq  \frac 32$, in both cases
we  have 
\[
(m^{1/q^*}\Log^{-(\gamman+1)q/q^*}(mn))^{(2-q^*)/2} \gtrsim_{\alpha}  n^{\beta}.
\]
Therefore, inequality \eqref{eq:lem-S-bdd-T-bddsupp} from Lemma~\ref{lem:S-bdd-T-bddsupp} 
(applied with $a=m^{-1/q^*}\Log^{(\gamman+1)q/q^*}(mn)$ and $k=n$) yields 
\begin{align}\label{eq:lem-q-ge-8beta-ge-p*-3}
\Ex\sup_{s\in S_3, t \in B_p^n} \sum_{i\le m, j\le n} X_{i,j}s_it_j 
 \lesssim_{\alpha} m^{1/q} \sup_{t\in B_p^n} \Bigl\|\sum_{j=1}^n X_{1,j}t_j \Bigr\|_q
 + n^{(1/p^*-1/2)\vee 0}. 
\end{align}

Since
\[
n^{(1/p^*-1/2)\vee 0} 
= \sup_{t\in B_p^n} \|t\|_2 
\le\sup_{t\in B_p^n}\Bigl\|\sum_{j=1}^n X_{1,j}t_j \Bigr\|_q,
\]
estimates \eqref{eq:lem-q-ge-8beta-ge-p*-1}-\eqref{eq:lem-q-ge-8beta-ge-p*-3} 
yield the assertion.
\end{proof}


\subsection{Case $ 24\gamman \geq q \geq p^*$}

Once we prove the upper bound in the case $ 24\gamman \geq q \geq p^*$, 
the upper bound in the case $ 24\gamman \geq p^* \geq q$ 
follows by duality   \eqref{eq:duality}. 
We first deal with the case $p^*\ge 2$ and then move to the case $2\ge p^*$ at the end of 
this subsection.

Let us begin with the proof in the case $p^*=q\ge 2$, when an interpolation argument works.
	
\begin{lem}	
\label{lem:q=p*}
If $p^*=q\ge 2$, then the upper bound in Theorem~\ref{thm:rect} holds.
\end{lem}

\begin{proof}
By Subsections~\ref{seq:case-leq2} and \ref{sect:above-4} we know that the assertion holds 
when $p^*=q\in\{2\} \cup[3,\infty]$. 

Assume without loss of generality that $\Ex X_{i,j}^2=1$. 
Fix $p^*=q\in (2,3)$ and let $\theta\in (0,1)$ be such that $\frac 1q = \frac{\theta}2+\frac{1-\theta}{3}$, 
i.e., $\frac 1p = 1-\frac 1q = \frac{\theta}2 +\frac{1-\theta}{3^*}$. 
Then  \eqref{eq:compmomregalpha} implies that
\begin{equation}	
\label{eq:lem-q=p*}
\sup_{t\in B_p^n}\Bigl\|\sum_{j=1}^nt_jX_{1,j}\Bigr\|_{ q\wedge \Log m} 
\sim_\alpha \sup_{t\in B_p^n}\Bigl\|\sum_{j=1}^nt_jX_{1,j}\Bigr\|_{ 2} =1 ,
\end{equation}
and similarly
\begin{equation}	
\label{eq:lem-q=p*-2}
\sup_{s\in B_{q^*}^m}\Bigl\|\sum_{i=1}^ms_iX_{i,1}\Bigr\|_{ p^*\wedge \Log n} \sim_\alpha  1 ,
\end{equation}
By the Riesz-Thorin interpolation theorem, H{\"o}lder's inequality, \eqref{eq:lem-q=p*} and \eqref{eq:lem-q=p*-2} we get
\begin{align*}
\Ex\bigl\|(X_{i,j})_{i,j}\bigr\|_{\ell_p^n\rightarrow\ell_q^m}
& \le
\Ex \bigl(\bigl\|(X_{i,j})_{i,j}\bigr\|_{\ell_{2}^n\rightarrow\ell_2^m}^{\theta}
\bigl\|(X_{i,j})_{i,j}\bigr\|_{\ell_{3^*}^n\rightarrow\ell_{3}^m}^{1-\theta} \bigr)
\\ 
&\le 
\Bigl(\Ex\bigl\|(X_{i,j})_{i,j}\bigr\|_{\ell_{2}^n\rightarrow\ell_2^m}\Bigr)^{\theta}
\Bigl( \Ex \bigl\|(X_{i,j})_{i,j}\bigr\|_{\ell_{3^*}^n\rightarrow\ell_{3}^m}\Bigr)^{1-\theta}
\\
& \lesssim_{\alpha} 
(n \vee m)^{\theta/2}(n\vee m)^{(1-\theta)/3} 
= (n\vee m)^{1/q} \sim n^{1/p^*}+m^{1/q}.
\qedhere
\end{align*}
\end{proof}

\begin{proof}[Proof of the upper bound in Theorem~\ref{thm:rect} in the case $ 24\gamman \geq q \geq p^* \geq 2$]
 By Remark~\ref{rem:symmetrization} it suffices to assume that  $X_{i,j}$'s are symmetric and $\alpha\ge  \sqrt{2}$. 
 Then $\gamman=\log_2\alpha\ge  1/2$.
  Inequality \eqref{eq:compmomregalpha} implies that in the case $ 24\gamman \geq q \geq p^* \geq 2$ the upper bound in Theorem~\ref{thm:rect}  is  equivalent to 
\begin{equation}
\label{eq:est2show}
\Ex\|(X_{i,j})_{i,j}\|_{\ell_{p}^n\to\ell_q^n}
\lesssim_\alpha n^{1/p^*}+m^{1/q}.
\end{equation}

Observe first that we may assume that $m\geq n$.
Indeed, if $m\leq n$ then Lemma~\ref{lem:q=p*} yields
\[
\Ex\|(X_{i,j})_{i,j}\|_{\ell_{p}^n\to\ell_q^m}
\leq \|\mathrm{Id}\|_{\ell_{p}^n\to\ell_{q^*}^n}\Ex\|(X_{i,j})_{i,j}\|_{\ell_{q^*}^n\to\ell_q^m} 
\leq n^{1/q^*-1/p}\Ex\|(X_{i,j})_{i,j}\|_{\ell_{q^*}^n\to\ell_q^{ m}}
\lesssim
 n^{1/p^*}.
\]

Thus, in the sequel we assume that $2\leq p^*\leq q\leq  24\gamman$ and $m\geq n$. 
Define
\[
k_0:=\inf\biggl\{k \in \{0,1,\ldots\}\colon\ 
2^k\geq \frac{5}{\ln \alpha}\frac{2-q^*}{q^*}\Log m\biggr\}. 
\]
 Observe that
\begin{equation}
\label{eq:estk_0}
k_0=0 \quad \mbox{or}\quad 2^{k_0}\leq \frac{10}{\ln \alpha}\frac{2-q^*}{q^*}\Log m.
\end{equation}

By Lemma \ref{lem:vial2tol2} and the definition of $k_0$ we have
\begin{align*}
\Ex\bigl\|\bigl(X_{i,j}I_{\{|X_{i,j}|\geq \alpha^{k_0}\}}\bigr)_{i\leq m,j\leq n}\bigr\|_{\ell_p^n\to\ell_q^m}
&\leq
\Ex\bigl\|\bigl(X_{i,j}I_{\{|X_{i,j}|\geq \alpha^{k_0}\}}\bigr)_{i\leq m,j\leq n}\bigr\|_{\ell_2^n\to\ell_2^m}
\\
&\lesssim_\alpha \sqrt{m}\exp\bigl(-\frac{\ln \alpha}{10} 2^{k_0}\bigr)
\\
&\leq m^{\frac{1}{2}-\frac{2-q^*}{2q^*}}=m^{1/q}.
\end{align*}

By Lemma~\ref{lem:comwithgauss2} and two-sided bound \eqref{eq:formula-Gaussians} we have
\[
\Ex\bigl\|\bigl(X_{i,j}I_{\{|X_{i,j}|\leq 1\}}\bigr)_{i\leq m,j\leq n}\bigr\|_{\ell_p^n\to\ell_q^m}
\lesssim \Ex\|(g_{i,j})_{i\leq m,j\leq n}\|_{\ell_p^n\to\ell_q^m}\lesssim_\alpha n^{1/p^*}+m^{1/q}.
\]

We have $B_{q^*}^m\subset S_1+S_2$, where
\[
S_1=\{s\in B_{q^*}^m\colon\ |\supp(s)|\leq m^{1/(2\gamman q)}\},
\quad S_2=B_{q^*}^m\cap m^{-1/(2\gamman qq^*)}B_\infty^m.
\] 
Inequality \eqref{eq:lem-T-bdd-S-bddsupp} from Lemma~\ref{lem:S-bdd-T-bddsupp} applied 
with $b=1$, $l=m^{1/(2\gamman q)}$
shows that
\begin{equation}\label{eq:proof-thm1-p*q-between2and8beta-1}
\Ex\sup_{ s\in S_1,t\in B_p^n}
\sum_{i\leq m,j\leq n}X_{i,j}s_it_j
\lesssim_\alpha n^{1/p^*}+m^{1/q}.
\end{equation}

Since $2\gamman qq^*\leq   100\gamman^2$  we have
\[
S_2\subset S_3:=B_{q^*}^m\cap m^{-1/(100\gamman^2)}B_\infty^m.
\]
 Thus, to finish the proof it is enough to upper bound  the following quantity
\begin{align*}
\Ex\sup_{s\in  S_3,t\in B_p^n}
\sum_{i\leq m,j\leq n}X_{i,j}I_{\{1\leq |X_{i,j}|<\alpha^{k_0}\}}s_it_j
\leq \sum_{k=1}^{k_0}\Ex\sup_{s\in  S_3,t\in B_p^n}
\sum_{i\leq m,j\leq n}X_{i,j}I_{\{\alpha^{k-1}\leq |X_{i,j}|<\alpha^{k}\}}s_it_j.
\end{align*}

Let $u_1,\ldots,u_{k_0}$ be positive numbers to be chosen later.  
We decompose the set $ S_3$ in the following way, depending on $k$:
\[
S_3\subset S_{ 4,k}+S_{ 5,k},
\]
where
\[
S_{ 4,k}:=
\{s\in B_{q^*}^m\colon\ |\supp s|\leq m/u_k\}\cap m^{-1/( 100\gamman^2)}B_\infty^m,
\quad
S_{ 5,k}:=B_{q^*}^m\cap \Bigl(\frac{u_k}{m}\Bigr)^{1/q^*}B_\infty^m.
\]
Thus,
\begin{multline*}
\Ex\sup_{s\in { S_3}, t\in B_p^n}
\sum_{i\leq m,j\leq n}X_{i,j}I_{\{\alpha^{k-1}\leq|X_{i,j}|<\alpha^{k}\}}s_it_j
\\
\leq \Ex\sup_{s\in S_{ 4,k}, t\in B_p^n}
\sum_{i\leq m,j\leq n}X_{i,j}I_{\{|X_{i,j}|<\alpha^{k}\}}s_it_j
+\Ex\sup_{s\in S_{ 5,k}, t\in B_p^n}
\sum_{i\leq m,j\leq n}X_{i,j}I_{\{\alpha^{k-1}\leq |X_{i,j}|\}}s_it_j.
\end{multline*}

Observe that $B_p^n\subset B_2^n$ and
\[
\sup_{s\in S_{ 5,k}}\|s\|_2
\leq \sup_{s\in S_{ 5,k}}\|s\|_{q^*}^{q^*/2}\|s\|_\infty^{(2-q^*)/2}\leq \Bigl(\frac{u_k}{m}\Bigr)^{\frac{2-q^*}{2q^*}}.
\]
Hence, Lemma \ref{lem:vial2tol2} yields
\begin{align*}
\Ex\sup_{s\in S_{ 5,k}, t\in B_p^n} \sum_{i\leq m,j\leq n}X_{i,j}I_{\{\alpha^{k-1}\leq |X_{i,j}|\}}s_it_j
&\leq 
\Bigl(\frac{u_k}{m}\Bigr)^{\frac{2-q^*}{2q^*}}
\Ex\|(X_{i,j}I_{\{\alpha^{k-1}\leq |X_{i,j}|\}})\|_{\ell_2^m\to\ell_2^n}
\\
&\lesssim_\alpha m^{1/q}u_k^{\frac{2-q^*}{2q^*}}\exp\Bigl(-\frac{\ln \alpha}{10}2^{k-1}\Bigr).
\end{align*}
Thus, if we choose
\[
u_k:=\exp\Bigl(\frac{q^*\ln\alpha}{20(2-q^*)}2^k\Bigr),
\]
we get 
\[
\sum_{k=1}^{k_0}
\Ex\sup_{s\in S_{ 5,k}, t\in B_p^n}
\sum_{i\leq m,j\leq n}X_{i,j}I_{\{\alpha^{k-1}\leq |X_{i,j}|\}}s_it_j
\lesssim_\alpha \sum_{k = 1}^\infty m^{1/q}\exp\Bigl(-\frac{\ln \alpha}{40}2^k\Bigr)
\lesssim_\alpha m^{1/q}.
\]

Lemmas \ref{lem:comwithgauss2} and \ref{lem:viaChevet2} applied with $l=\frac m{u_k}$, 
$a=m^{-1/( 100\gamman^2)}$, and  $\gamma=24\gamman$ yield
\[
\Ex\sup_{s\in S_{ 4,k},t\in B_p^n}\sum_{i\leq m,j\leq n}X_{i,j}I_{\{|X_{i,j}|<\alpha^{k}\}}s_it_j
\lesssim_\alpha 
\alpha^k\Bigl(m^{-\frac{(2-q^*)}{ 200\gamman^2}}n^{1/p^*}+\sqrt{\Log u_k}(m/u_k)^{1/q}\Bigr).
\]
Property  \eqref{eq:estk_0}
yields
\begin{align*}
\sum_{k=1}^{k_0} \alpha^k m^{-\frac{(2-q^*)}{ 200\gamman^2}}n^{1/p^*}
&
\lesssim \alpha^{k_0} I_{\{k_0\neq 0\}} m^{-\frac{(2-q^*)}{ 200\gamman^2}}n^{1/p^*} 
\lesssim_{\alpha}
\Bigl(\frac{10}{\ln \alpha}\frac{2-q^*}{q^*} \Log m\Bigr)^{\log_2\alpha}
e^{-\frac{(2-q^*)}{ 200\gamman^2}\ln m}n^{1/p^*}
\\
&\lesssim_\alpha  n^{1/p^*}\sup_{x>0}x^{\log_2\alpha}e^{-x}\lesssim_\alpha n^{1/p^*}.
\end{align*}

Finally, since $q\leq  24\gamman$  and $u_k\geq 1$ we get 
$\sqrt{ \Log u_k}(m/u_k)^{1/q}\lesssim_\alpha m^{1/q}u_k^{-1/(2q)}$, so
\[
\sum_{k=1}^{k_0} \alpha^k\sqrt{\Log u_k}(m/u_k)^{1/q}
\lesssim_\alpha m^{1/q}\sum_{k\geq 1}\alpha^k\exp\Bigl(-\frac{q^*\ln\alpha}{40q(2-q^*)}2^k\Bigr)
\lesssim_\alpha m^{1/q}.
\qedhere
\]
\end{proof}

The proof in the case $ 24\gamman \geq q  \geq 2 \geq p^*$ is easy and bases on the already proven case when  $q\ge 2=p^*$ (see the proof above).

\begin{proof}[Proof of the upper bound in Theorem~\ref{thm:rect} in the case $ 24\gamman \geq q \geq 2 \geq p^*$]
 Inequality \eqref{eq:compmomregalpha} implies that in the case $ 24\gamman \geq q \geq  2  \geq p^*$ the upper bound in Theorem~\ref{thm:rect}  is  equivalent to 
\begin{equation}
\label{eq:est2show-under2-2}
\Ex\|(X_{i,j})_{i,j}\|_{\ell_{p}^n\to\ell_q^m}
\lesssim_\alpha n^{1/p^*}+m^{1/q}n^{1/p^*-1/2}.
\end{equation}
In particular, an already obtained 
 upper bound in the case $ 24\gamman\geq q\ge 2=p^*$ yields
 \begin{equation*}
\label{eq:est2show-under2-1}
\Ex\|(X_{i,j})_{i,j}\|_{\ell_{2}^n\to\ell_q^m}
\lesssim_\alpha n^{1/2}+m^{1/q},
\end{equation*}
 so
\begin{align*}
\Ex\|(X_{i,j})_{i,j}\|_{\ell_{p}^n\to\ell_q^m} 
&\le
\|\operatorname{Id}\|_{\ell_p^n \to \ell_2^n} \Ex\|(X_{i,j})_{i,j}\|_{\ell_{2}^n\to\ell_q^m}
\lesssim_{\alpha} n^{1/p^*-1/2}(n^{1/2}+m^{1/q}) 
\\
&= n^{1/ p^*}+m^{1/q} n^{1/p^*-1/2},
\end{align*}
and thus, \eqref{eq:est2show-under2-2} holds.
\end{proof}


  \bibliographystyle{amsplain}
  \bibliography{matrices}

\end{document}